\documentclass[12pt,oneside,reqno]{amsart}

\usepackage{comment}
\usepackage{mathtools,nccmath}
\usepackage{afterpage}
\usepackage{etoolbox}
\usepackage[english]{babel}
\usepackage{amssymb,amsmath,amsthm}
\usepackage{bbm}
\usepackage{fancyhdr}
\usepackage{enumerate}
\usepackage{ifthen}
\usepackage{color}
\provideboolean{shownotes} 
\setboolean{shownotes}{true}
\usepackage{hyperref}
\usepackage{graphicx}
\graphicspath{ {./images/} }

\newcommand{\margnote}[1]{
	\ifthenelse{\boolean{shownotes}}%
	{\marginpar{\raggedright\tiny\texttt{#1}}}%
	{}%
}

\newcommand{\hole}[1]{
	\ifthenelse{\boolean{shownotes}}%
	{\begin{center} \fbox{ \rule {.25cm}{0cm}
				\rule[-.1cm]{0cm}{.4cm} \parbox{.85\textwidth}{\begin{center}
						\texttt{#1}\end{center}} \rule {.25cm}{0cm}}\end{center}}
	{}
}
\theoremstyle{plain} \newtheorem{Theorem}{Theorem}
\theoremstyle{plain} 
\theoremstyle{plain} \newtheorem{Lemma}{Lemma}
\theoremstyle{plain} \newtheorem{Proposition}{Proposition}
\theoremstyle{definition} \newtheorem{Definition}{Definition}
\theoremstyle{definition}
\theoremstyle{definition}
\theoremstyle{definition} 

\newcommand{\Rd}{{\mathbb{R}^d}}

\newcommand{\Zz}{{\mathbb{Z}}}

\newcommand{\vsp}[1]{\vspace{0.5cm}\par}

\numberwithin{equation}{section}
\numberwithin{Theorem}{section}
\numberwithin{Proposition}{section}
\numberwithin{Lemma}{section}
\numberwithin{Definition}{section}
\numberwithin{Remark}{section}

\newcommand{\ra}{\rangle}
\newcommand{\la}{\langle}

\allowdisplaybreaks

\usepackage{amssymb,amsmath,bm}
\usepackage{graphicx}
\usepackage{caption}
\usepackage{color}
\usepackage[usenames,dvipsnames]{xcolor}
\usepackage{tikz}
\newtheorem{proposition}{Proposition}[section]
\newtheorem{theorem}[proposition]{Theorem}

\newtheorem{lemma}[proposition]{Lemma}

\numberwithin{equation}{section}
\newcommand{\epsi}{\varepsilon}
\newcommand{\UUU}{\color{black}}

\newcommand{\EEE}{\color{black}}

\newcommand{\Nz}{{\mathbb N}}
\newcommand{\Rz}{{\mathbb R}}
\newcommand{\Rzn}{{\mathbb R}^{3 \times 3}}
\newcommand{\Rzd}{\Rz^{3\times 3}_{\rm dev}}
\newcommand{\Rzs}{\Rz^{3\times 3}_{\rm sym}}
\newcommand{\Rzsp}{\Rz^{3\times 3}_{\rm sym+}}
\newcommand{\haz}{\widehat}

\renewcommand{\d}{{\rm d}}
\newcommand{\SL}{\text{\rm SL}(3)}

\newcommand{\SO}{\text{\rm SO}(3)}
 
\newcommand{\GLp}{\text{\rm GL}^+(3)}
\newcommand{\GLps}{\text{\rm GL}^+_{\rm sym}(3)}

\newcommand{\weak}{\rightharpoonup}

\newcommand{\tr}{\textrm{tr}\,}
\newcommand{\cof}{\textrm{cof\,}}
\newcommand{\dev}{{\rm dev}\,} 
\newcommand{\TT}{\top}

\newcommand{\ove}{\overline}

\def\We{W_{ e}}
\def\Wh{W_{  p}}

\def\bP{P} 
\def\bPi{\Pi} 
\def\bF{ F} 
\def\bN{N} 
 
\def\bFe{F_{\rm e}}
\def\one{{I}}
\def\bbC{\mathbb{C}}

\def\bbG{\mathbb{G}}

\title[Nonassociative finite plasticity]
{Quasistatic nonassociative plasticity at finite strains} 

\author[U. Stefanelli]{Ulisse Stefanelli} 
	\address[Ulisse Stefanelli]{University of
		Vienna, Faculty of Mathematics,
                Oskar-Morgenstern-Platz 1, A-1090 Vienna, Austria. 
		University of Vienna, Vienna Research Platform on Accelerating
		Photoreaction Discovery, W\"ahringerstra\ss e 17, 1090 Wien, Austria.
	 Istituto di	Matematica Applicata e Tecnologie Informatiche {\it E. Magenes}, via
		Ferrata 1, I-27100 Pavia, Italy.
	}
	\email{ulisse.stefanelli@univie.ac.at}
	\urladdr{http://www.mat.univie.ac.at/$\sim$stefanelli}
        \author[A. Vikelis] {Andreas  Vikelis}
\address[Andreas  Vikelis]{University of Vienna,  Faculty of Mathematics, 
Oskar-Morgenstern-Platz 1, A-1090 Vienna, Austria. Czech Technical University in Prague, Faculty of Civil Engineering, Department of Mechanics, Th\'akurova 7, Prague 6, 166 29, Czech Republic.}
\email{andreas.panagiotis.vikelis@univie.ac.at}
{ }
\subjclass[2010]{49J45, 49S05, 74C15}
\keywords{Finite-strain plasticity,  
  quasistatic evolution, energetic solutions, small-deformation limit}
\begin{document}
\begin{abstract}
We investigate finite-strain elastoplastic evolution in
the nonassociative setting. The constitutive material model is
formulated in variational terms and coupled with the quasistatic
equilibrium system. We introduce measure-valued energetic solutions and prove
their existence via a time discretization approach. The existence
theory hinges on a suitable regularization of the dissipation term via
a space-time mollification. Eventually, we discuss the possibility of
solving the problem in the setting of functions, instead of measures. 
\end{abstract}
\maketitle

\section{Introduction}

The {\it state} of an elastoplastic material  at finite strains is described
in terms of its {\it deformation} $y:\Omega \to \Rz^3 $ with respect
to the reference configuration $\Omega \subset \Rz^3$ and of
the {\it plastic  strain} $\bP : \Omega \to
\SL:=\{ A \in \Rzn \;| \, {\rm det} \,A=1\}$, encoding the
information on previous plastic transformations.
Evolution is governed by the competition between
energy-storage and dissipation mechanisms. 
 We assume the classical {\it multiplicative
  decomposition} $\nabla y =F_e P$, where $F_e$ stands for the
{\it elastic deformation tensor},
see \cite{Kroener,Lee}.  The stored energy of the medium 
is  assumed to have the form
\begin{equation}
\mathcal{W}(y,P)=\int_\Omega \We( \nabla y
 \bP^{-1}) \,\d x+ \int_\Omega\Wh(\bP)\, \d x + \frac{\mu}{ q_ r }\int_\Omega |\nabla \bP |^{ q_ r }
 \, \d x. \label{eq:WW}
\end{equation}
Here,  $W_e$ and $W_p$ stand for the {\it elastic} and the {\it plastic energy
density}, respectively, while the gradient term features an additional length-scale
parameter $\mu$, in the spirit of so-called {\it gradient-plasticity}
theories \cite{Fleck,Fleck2,Aifantis}.

The variations of the energy with respect to the variables $\nabla y$
and $P$ correspond to the (generalized) forces driving the
evolution. In particular, $\Pi=\partial_{\nabla y}\mathcal{W}$ is the
first Piola-Kirchhoff stress, which is assumed to fulfill the
quasistatic system $-{\rm div}\, \Pi =b$, for a given force density
$b$. On the other hand, the thermodynamic force
$N=-\partial_P\mathcal{W}$ associated to $P$ fulfill the {\it plastic
  flow rule}, expressed in complementarity form as
\begin{equation}\dot P = \zeta \, \partial_N g(P,N),\quad \zeta\geq 0, \quad
f(P,N)\leq 0, \quad \zeta \, f(P,N)=0.\label{eq:zeta}
\end{equation}
Here, $f$ is the so-called {\it yield function} and its level set
$\{f(P,\cdot)\leq 0\}$ is the {\it elastic domain}. The potential $g$
instead specifies the plasticization direction via its gradient
$\partial_N g(P,N)$.

The {\it associative} case corresponds to the
choice $f\equiv g$ in \eqref{eq:zeta} and has already been
extensively  studied.  The static problem for $\mu=0$ has been firstly considered
in \cite{Mielke04b} under some
restrictions on the choice of the driving functionals.  The incremental problem has been addressed in
\cite{compos} for $\mu=0$ and in \cite{Mielke2005} by including some
gradient penalization, i.e., $\mu>0$, although
in terms of ${\rm curl}\, P$ only. The existence of quasistatic
evolution in terms of {\it energetic solutions} \cite{Mielke2015} had been
tackled in \cite{compos2} for $\mu=0$ and in \cite{Mainik} for
$\mu>0$ and $ q_ r =2$. Linearization for small strains  can be found
in 
\cite{ms5}. For $\mu>0$, a number of additional results are available,
ranging from the symmetric formulation in terms of
$P^\top P$ 
\cite{cplas1,cplas2}, the alternative multiplicative
decomposition $\nabla y =PF_e$ \cite{Davoli}, shape memory alloys
\cite{ams,Grandi2}, dimension reduction \cite{davoli.mora},
viscoplasticity \cite{MRS_vis}, and numerical
approximations~\cite{Mielke16}. A recent approach, focusing on
modeling the evolution of dislocation structures is in
\cite{Rindler1,Rindler2}. \UUU See also \cite{Conti1,Conti2} for some
model of dislocations concentrating on curves, \cite{Garroni,Ginster}
and \cite{Scardia1, Scardia2, Ginster1} for the derivation of a 
strain-gradient model via dislocation upscaling, in the linearized and
in the nonlinear setting, respectively, and \cite{Kupferman} for another
result on dislocation homogenization. 
\EEE

The focus of this paper is on the general  {\it nonassociative}
case instead, where $f\not \equiv g$. To the best of our knowledge, in
the nonassociative setting the only available 
quasistatic
existence results are limited to the linearized, infinitesimal strain
case \cite{Mora,Mora2}. Here the geometrically and elastically
nonlinear case
is considered  instead.  Our main result is the existence of energetic
solutions for the quasistatic evolution problem. This is classically obtained via
a time-discretization approach, based on a sequence of incremental
problems.

Compared with the classical associative case, the nonassociative
setting presents some distinctive additional intricacy. First and
foremost, the mismatch $f\not \equiv g$ in the flow rule
\eqref{eq:zeta} calls for some specific variational reformulation of
the constitutive equation. This in
turn makes dissipative dynamics nonlinearly dependent on the current deformation
state $\nabla y$. In the frame of \eqref{eq:WW}, the variable
$\nabla y$ lacks strong compactness, one is then forced to regularize
the model by regularizing the dependence of the dissipation on $\nabla
y$ via a space-time mollification procedure. In fact, such
mollification was needed in the linearized case  \cite{Mora,Mora2}, as
well. On the other hand, in the finite-strain setting one is
additionally forced to resort to measure-valued solutions.

The paper is organized as follows. After fixing some notation in
Section \ref{sec:notation}, we introduce and discuss the constitutive
elastoplastic material model in Section \ref{sec:const}. The full
quasistatic evolution problem is presented in Section \ref{sec:newq},
together with a discussion on the simplifications and regularizations
that are needed for the analysis. In Section \ref{sec:energy}, we give
the  details  of the variational formulation, present assumptions, define
measure-valued solutions, and state the main existence result, namely,
Theorem \ref{thm:existencemv}. In addition, a statement in terms of
functions, Proposition \ref{lemma:functions} is presented. The proofs of Theorem \ref{thm:existencemv} and Proposition
\ref{lemma:functions} are in Section \ref{sec:proof} and
\ref{proof:lemma:functions}, respectively.
Eventually, the Appendix contains the statement and the proof of an
extended Helly Selection Principle, which is used in the existence argument.

%
\section{Notation}\label{sec:notation}
%
In this section we fix the notation used in the paper and we present
our nonassociative elastoplastic model at finite strains. 

\subsection{Tensors and function spaces}
%
 In our analysis, we focus  on the three-dimensional setting and denote the space of
$2$-, $3$-, and $4$-tensors in $\Rz^3$ by
$\Rz^{3\times 3}$, $\Rz^{3\times 3\times 3}$, and $\Rz^{3\times
  3\times 3 \times 3}$, respectively. For all $  A\in \Rz^{3\times 3}$ we define the {\it
  trace} as $\tr  A : = A_{ii}$
(summation convention on repeated indices), the {\it deviatoric part} as
${\rm dev} A  =  A - (\tr  A) \one
/3$ where $\one$ is the identity $2$-tensor, and the (Frobenius)
 norm as $| A|^2 := \tr ( A ^\TT A)
 $ where the symbol $\TT$ denotes transposition. The contraction product between $2$-tensors is $A {:}
  B  := A_{ij}B_{ij}$ and we classically denote the
 scalar product of vectors in $\Rd$ by $a{\cdot} b:=a_ib_i$. The
 symbols $\Rzs$ and $\Rzsp$ stand for the  subsets  of $\Rz^{3 \times
   3}$ of symmetric tensors and   of  symmetric positive-definite tensors,
 respectively. Moreover, $\Rzd$ indicates the space of symmetric
 deviatoric tensors, namely $\Rzd:=\{A \in  \Rzs  \; |\; \tr
  A =0\}$.  The symbol $A^s=(A+A^\TT)/2$ denotes the symmetric
  part of the tensor $A\in
  \Rz^{3\times 3}$.   We shall use also the following tensor sets
\begin{align*}
  &\SL:=\{ A \in \Rzn \;|\; \det  A
  =1\}, &&\SO:=\{ A  \in  \SL  \;|\;  A^{-1} =
           A ^\TT  \},\\
  & \GLp:=\{ A  \in \Rzn \;|\; \det  A >0 \},  &&\GLps:= \GLp\cap \Rzs.
\nonumber
\end{align*}
The tensor $\cof A$ is the {\it cofactor matrix} of $A $. For $A$ 
invertible we have that $\cof A=(\det A)\,A^{-\TT}$.

The contraction between two 3-tensors $A,\, B \in \Rz^{3\times 3
  \times 3}$ is defined as $A \vdots B = A_{ijk} B_{ijk}$ and the
product between a 4-tensor $\bbC \in \Rz^{3\times 3
  \times 3\times 3}$ and a 2-tensor $D\in \Rz^{3\times 3}$ is defined
componentwise as $\bbC D_{ij} = \bbC_{ij\ell k} D_{\ell k}$.
We use the same symbol $|\cdot|$ to denote the Frobenius norm of 3- and
4-tensors.

In the following, we denote by $\partial \varphi$ the subdifferential
of the smooth or  of the  convex, proper, and lower semicontinuous function
$\varphi: E \to (-\infty,\infty]$ where $E$ is a normed space with dual
$E^*$ and duality pairing $\langle \cdot, \cdot \rangle $
\cite{Brezis73}. In particular, $ y^* \in \partial \varphi (x)$ iff
$\varphi(x) <\infty$ and 
$$ \langle y^*,w{-}x\rangle \leq \varphi(w) - \varphi(x) \quad
\forall w \in E.$$

The symbol $\|\cdot \|_E$ denotes the norm in the normed space $E$. We
use the classical notation $L^p$ and
$W^{1,r}$  for Lebesgue and Sobolev spaces. Whenever clear from the context, we simply indicate their
norms by $\| \cdot \|_p$ and $\| \cdot \|_{1,r}$, without specifying
target spaces. The symbols $L^p(0,T;E)$, $C([0,T];E)$, and $C^1([0,T];E)$ denote
classical Bochner spaces of Lebesgue $p$-integrable, continuous on
$[0,T]$, and continuously differentiable functions of time with
values in $E$.



We use the symbol $\nabla$ to indicate the Jacobian of a tensor
valued-function on the open set $\Omega\subset \Rz^3$. In particular, for $y :
\Omega \to \Rz^3$ and $P:\Omega \to \Rz^{3 \times 3}$ we have
 $(\nabla y)_{ij} = \partial y_i /\partial x_j$ and $(\nabla
 P)_{ijk} = \partial P_{ij}/\partial x_k$.
 

\subsection{A caveat on notation}
In the following, we use the same symbol $c$ to indicate a generic positive
constant depending only on data. Note that the value of $c$ can change
from line to line.

\section{Constitutive model}\label{sec:const}

This section is devoted to the introduction of the constitutive
model. 
We denote by $F \in \GLp$ and $P \in \SL$ the {\it deformation gradient} and
the {\it plastic
strain}, respectively, and we assume plasticity
to be isochoric, namely, $\det P=1$. The classical
multiplicative decomposition~\cite{Kroener,Lee} 
\begin{align*}
    F=F_e P
\end{align*}
is postulated, 
where $F_e\in  {\rm GL}^+(3)$ denotes the {\it elastic strain}. 

\subsection{Energy}
The evolution of the elastoplastic medium results from the
interplay between energy-storage mechanisms and plastic-dissipation
effects. The {\it local energy density} $W$ of the medium is assumed to be
additively decomposed as 
$$W(F,P) =  W_e( F_e) + W_p(P)= W_e( F
P^{-1}) + W_p(P),$$
namely, into an elastic and a plastic contribution, indicated by $W_e$ and $W_p$,
respectively. Both the 
{\it elastic energy density} $W_e: {\rm GL}^+(3)\to [0,\infty)$
and the {\it plastic energy density} $W_p: \SL \to [0,\infty)$ are asked to be smooth. Moreover, $W_e$ is
assumed to be {\it frame indifferent}, i.e., $W_e(RF_e)=W_e(F_e)$
for all $F_e\in \GLp$ and $R\in {\rm SO}(3)$.

\subsection{Constitutive equations}
Following the classical Coleman-Noll procedure \cite{Coleman1974}, we
introduce the {\it constitutive relations} which identify the
thermodynamical forces associated with the deformation gradient $F$
and the plastic strain $P$. The evolution of $F$ is driven by the
first Piola-Kirchhoff stress $\Pi$ given by
\begin{equation}\label{eqn:piola}
    \bPi  :=\partial_{\bF}W(\bF,\bP)= \partial_{\bFe} \We(\bF \bP^{-1}) \bP^{-\TT},
\end{equation}
while the thermodynamic force conjugated to $P$ reads 
\begin{align}\label{eq:N}
\bN &:=-\partial_{\bP}W(\bF,\bP)=-\partial_{\bFe} 
      \We(\bFe):\partial_{\bP}\bFe-\partial_{\bP}\Wh(\bP)\nonumber\\
  &= \bP^{-\TT}
\bF^\TT \partial_{\bFe} 
\We(\bFe) \bP^{-\TT}-\partial_{\bP}\Wh(\bP).
\end{align}


\subsection{Flow rule} Plastic activation is classically described by
specifying the  {\it yield  function}
\begin{align*}
  &f =f(\bP,\bN): \SL\times \Rzn
    \to \Rz, \ \text{with} \  \\
  &\quad \bN \mapsto
f(\bP,\bN) \ \text{convex and} \ f(\bP,\, 0)<0\ \text{for all} \ \bP
\in \SL.\end{align*}
In
particular, the {\it elastic
domain} corresponds to the sublevel $\{f \leq 0\}$.

The direction of plastic trasformation is prescribed in terms of a
second {\it plastic potential}
\begin{align*}
  &g =g(\bP,\bN): \SL\times \Rzn
    \to \Rz, \ \text{with} \  \\
  &\quad \bN \mapsto
g(\bP,\bN) \ \text{convex and} \ g(\bP,\, 0) =  0\ \text{for all} \ \bP
    \in \SL.\end{align*}

  Moving from  the conjugacy of $N$ and
$P$ in \eqref{eq:N}, the plastic flow rule is expressed in
complementary form as 
\begin{equation}
  \label{eq:complementary}
  \dot P = \zeta \,\partial_{\bN} g(\bP,\bN), \quad \zeta \geq 0, \quad
  f(\bP,\bN) \leq 0, \quad \zeta f(\bP,\bN) = 0.
\end{equation}
The occurrence of the two distinct potentials $f$ and $g$ is the
specific trait of nonassociative plasticity: the plastic flow at $(P,N)$ is
activated in terms of $f$ while the plastic evolution $\dot P$ follows
the normal cone of the convex sublevel $\{\haz N \in \Rz^{3\times 3}\:
| \: g(P,\haz N)\leq g(P,N)\}$. 

In order to provide a compact reformulation of  the complementary conditions
\eqref{eq:complementary} in subdifferential form, we follow the
analysis in \cite{Laborde1,Laborde2}, see also \cite{Mora,Ulloa2021},
and define for all $(P,N)\in
\SL \times \Rz^{3\times 3}$ the function
\begin{equation}
r(P,N):= g(P,N) - f(P,N),\label{eq:ar}
\end{equation} 
as well as the
(possibly empty) convex set
$$L(P,N):=\{ \haz
  N\in \Rz^{3\times 3} \: |  \: g(P, \haz N) \leq
                                                    r(P,N)\}.$$
                                                    Note that we have
\cite{Laborde1,Laborde2}
         $$ N \in \{ f(P,\cdot) \leq 0\} \ \Leftrightarrow \ N \in
         L(P,N)$$
so that the set $L(P,N)$ is surely nonempty for all $N$ belonging to the elastic domain $\{ f(P,\cdot) \leq 0\}$. 
 The complementarity conditions
\eqref{eq:complementary} can be equivalently rewritten as
$$\dot P \in \partial I_{L(P,N)}(N)$$
where $ I_{L(P,N)}$ stands for the indicator function of the convex
set $L(P,N)$.
By passing to its equivalent dual form we get 
\begin{equation}
\bN \in \partial_{\dot P}  R(\bP, \bN , \dot
P) \label{combine1} 
\end{equation}
where the state-dependent  {\it infinitesimal
  dissipation} $R:\SL \times \Rz^{3\times 3} 
\times \Rz^{3\times 3}\to [0,\infty]$ reads
\begin{equation}
  \label{eq:locdiss}
  R( P,  N, \dot P):= \sup \big\{ \dot P :\haz N \:
  |\: 
  \haz N\in L(P,N) \big\}.
\end{equation}
Taking into account the definition \eqref{eq:N} of $N$, the flow rule
can be expressed in the compact form  
\begin{align}
    \partial_{\dot P}  R(\bP, \bN , \dot P)+ \partial_{\bP}W(\bF,\bP) \ni 0.\label{eq:flow0}
\end{align}

%
\subsection{An example} \label{sec:vonMises} In order to make the
discussion more concrete and to motivate the assumptions below,
we present here an 
example of flow rule inspired by the
classical von Mises theory \cite{Gurtin2010,Mandel1972}. Let the yield
function $f:\SL\times\mathbb{R}^{3\times 3}\to \mathbb{R}$ and the
plastic potential $g:\SL\times\mathbb{R}^{3\times 3}\to \mathbb{R}$ be
defined as
\begin{align*}
  f(\bP,\bN) &= |\bbG\,\dev(\bN \bP^\TT)| - r_0,\\
  g(\bP,\bN) &= |\dev(\bN \bP^\TT)|, 
\end{align*}
respectively, where $r_0>0$ is given and the  4-tensor  $\bbG$ is positive definite. The function $r $ from \eqref{eq:ar} reads
$$r(\bP,\bN)= g(\bP,\bN) -f(\bP,\bN)= |\dev(\bN \bP^\TT)| -
|\bbG\,\dev(\bN \bP^\TT)| +r_0.  $$
By imposing $ |\bbG|\leq 1
$ have that 
\begin{equation}
  \label{eq:h}
  r_0 \leq r(\bP,\bN) \leq r_0 + (1+|\bbG|)\, |\dev(\bN \bP^\TT)| \quad
  \text{for all} \,\,(\bP,\bN) \in
\SL\times \Rzn.
\end{equation}
The  infinitesimal dissipation $R$ from \eqref{eq:locdiss} reads
\begin{align}
    R( P,  N, \dot P)&=\sup  \big\{\dot PP^{-1}:B \: |  \: |\dev B|\leq r(
    P, N) \big\}\nonumber\\
  &=
  \left\{
    \begin{array}{ll}
     r( P,  N) |
      \dot P \bP^{-1}| & \text{if} \ \ \tr(  \dot P \bP^{-1})=0\\[1mm]
      \infty&\text{otherwise}.
    \end{array}
    \right. 
  \label{eq:locadiss2}
\end{align}

\subsection{Dissipativity}
Let us now check that the constitutive equations
\eqref{eqn:piola}--\eqref{eq:N} together with the flow rule \eqref{combine1} give
rise to a dissipative model. To this aim, assuming sufficient smoothness we
compute the variation of the internal energy $t \mapsto W(t) - \bPi (t): \dot F(t)$ along a trajectory obtaining

\begin{align}
 & \frac{\d}{\d t} W(\bF,\bP) - \bPi : \dot F = \left(\partial_{\bF} W(\bF,\bP)
                                               - \bPi\right): \dot \bF
                                                -\bN : \dot
                                               P \nonumber\\
  &\quad=  -\bN : \dot
                                               P  =
    -R(\bP,\bN,\dot P) - I_{L(P,N)}(N) =   -R(\bP,\bN,\dot P) \leq 0,
\label{Clausius1}
\end{align} 
confirming that $R(\bP,\bN,\dot P)$ is the  (nonnegative) infinitesimal
dissipation in the system.

\subsection{Linearization}\label{sec:linearization}  In this
section, we show that the finite-strain model reduces to the
infinitesimal, linearized-strain one of \cite{Ulloa2021} in the small-loading
limit. To this end, for the sole scope of this section we strengthen the requests on energy densities and plastic
potentials as follows
\begin{align} 
  &W_e\in C^2(\GLp), \ W_p \in C^2(\SL),  \nonumber\\
  &\quad  \partial_{F_e}W_e(I)=\partial_PW_p(I)=0, \  \mathbb C : =
    \partial_{F_e}^2W_e(I)>0, \  \mathbb H : = \partial_{P}^2W_p(I)>0, \label{eq:Wc2}\\
  &f,\, g \  \ \text{continuous},\ \ N \mapsto f(P,N)-f(P,0) \ \ \text{and}  \ \  
    N \mapsto g(P,N) \nonumber\\
  &\qquad \text{positively $1$-homogeneous for all}
                                            \  P \in \SL. \label{eq:poshom}
\end{align}
Note that these are compatible with  the von Mises example of Section
\ref{sec:vonMises}.

Given the finite-strain deformation gradient $F$ and the plastic strain
$P$ we define their small-strain analogues $\eta\in \Rz^{3\times 3}$ and
$p\in \Rz^{3\times 3}$ by posing 
$$F=I+\epsi \eta  \ \ \text{and} \ \ P=\UUU \text{exp}(\epsi p)$$
\UUU where $\text{exp}$ stands for the exponential matrix and
$\epsi>0$ is assumed to be small.  As 
$P\in \SL$, for all $\epsi>0$ one has $1=\det P = e^{\epsi\, {\rm
    tr}\, p}$. This implies that \EEE ${\rm tr} \, p = 0$, namely, $p\in
\Rz^{3\times 3}_{\rm dev}$.

The
constitutive relation \eqref{eqn:piola} is rewritten in terms of $(\eta ,p)$ as
$$\Pi = \partial_FW(F,P) = \partial_{F_e}W_e\big((I+\epsi \eta) \UUU\,\text{exp}(-\epsi p)\EEE\big)\UUU\,\text{exp}(-\epsi p^\top)\EEE.$$
As $\epsi \to 0$ one has $(I+\epsi \eta) \,\UUU\text{exp}(-\epsi p)\EEE \sim (I+\epsi \eta) (I-\epsi p)\EEE= I + \epsi (\eta -p) - \epsi^2\eta p$, as well as $\UUU\text{exp}(-\epsi p^\top)\EEE\sim I- {\epsi} p^\TT$.
By Taylor expanding and using $\partial_{F_e}W_e(I)=0$ from \eqref{eq:Wc2}, one gets
\begin{align*}
  &\Pi \sim  \partial_{F_e}W_e(I + \epsi (\eta -p) - \epsi^2\eta p)
    \sim \epsi \mathbb C(\eta - p) = \epsi \mathbb C(\eta - p)^{s}.
\end{align*}
The fourth-order tensor  $\mathbb C$  plays the role of the linearized
elasticity tensor. It is major symmetric, i.e.,
$\mathbb C_{ijk\ell} = \mathbb C_{k\ell ij}$, as it
is a Hessian, and minor symmetric, $\mathbb C_{ijk\ell} =
\mathbb C_{ij\ell k}$,  as a consequence of frame
indifference.
The linearized stress $\sigma=\Pi/\epsi$ takes the 
form $\sigma = \mathbb C(\eta - p)^{s} $.

The thermodynamic force conjugated to $P=\UUU\,\text{exp}(\epsi p)\EEE$ reads from
\eqref{eq:N} as
\begin{align*}
  N&= \UUU\text{exp}(-\epsi p^\top)\EEE (I + \epsi \eta)^\TT
  \partial_{F_e}W_e\big((I+\epsi \eta) \UUU\text{exp}(-\epsi p)\EEE\big) \,\UUU\text{exp}(-\epsi p^\top)\EEE-
  \partial_PW_p(\UUU\text{exp}(\epsi p)\EEE)\\
   &\sim \epsi\mathbb C(\eta - p)^s - \epsi\mathbb H p.
\end{align*}
In particular, the linearized thermodynamic force $n=N/\epsi$ takes the form
$n=\sigma - \mathbb H p$. The fourth-order tensor $\mathbb H$ is
major symmetric and corresponds to the linearized kinetic hardening
tensor. 

In contrast to the analysis of the constitutive relations, the
linearization of the flow rule calls for rescaling the yield stress $f(P,0)$
in order to allow for plastic evolution in the small-strain limit. In
particular, we redefine
$$f_\epsi(P,N) = f(P,N) - f(P,0)+\epsi f(P,0)$$
Taking advantage of the positive homogeneity from
\eqref{eq:poshom} one checks that
\begin{align*}
  f_\epsi(P,\haz N)  = f_\epsi(\UUU \text{exp}(\epsi p)\EEE,
  \epsi \haz n)\leq 0 \
  \Leftrightarrow \  \epsi (f(\UUU \text{exp}(\epsi p)\EEE,
 \haz n)-f(\UUU \text{exp}(\epsi p)\EEE,0)) + \epsi f(\UUU \text{exp}(\epsi p)\EEE,0)\leq 0.
\end{align*}
In the limit $\epsi \to 0$ the rescaled elastic domain hence reads
$\{\haz n \in \Rz^{3\times 3}\, | \, f(I,\haz n) \leq 0\}$ and is \UUU
independent of \EEE
$p$. Correspondingly, by letting  
$$r_\epsi(P,N) = g(P,N)-f_\epsi(P,N) \ \ \text{and} \ \ L_\epsi(P,N) =
\{\haz N \in \Rz^{3\times 3}\, |\, g(P,\haz N) \leq r_\epsi(P,N)\}$$
we find
$$\haz N  = \epsi \haz n \in L_\epsi(P,N) \ \Leftrightarrow \
g(I+\epsi p,\haz n) \leq r(I+\epsi p, n).$$
By taking the limit $\epsi \to 0$, the linearization of the relation $\haz N \in
L_\epsi (P,N)$ reads $$\haz n \in L_0(n):=\{\haz n \in
\Rz^{3\times 3} \, | \, g(I,\haz n) \leq r(I,n)\}.$$
Note that the
latter set $L_0$ still depends on the linearized thermodynamic force $n$ but it
is independent of the linearized plastic strain $p$.

The infinitesimal dissipation is linearized as follows
\begin{align*}
  R(P,N,\dot P) &= R (\UUU \text{exp}(\epsi p)\EEE,\epsi n,\epsi \dot p) = \sup\{\epsi \dot
  p:\epsi \haz n\, | \,\epsi \haz n\in L_\epsi (\UUU \text{exp}(\epsi p)\EEE,\epsi n)\}\\
  &\sim \epsi^2 \sup \{\dot p:\haz n \,| \, \haz n \in
    L_0(n)\}=:\epsi^2R_0(n,\dot p)
\end{align*}
In the linearization limit, the infinitesimal dissipation turns out to
be independent of $p$. Eventually, the \UUU flow rule \EEE \eqref{eq:flow0} rewrites
$$0\in  \partial_{\dot P}  R(\bP, \bN , \dot P)+
\partial_{\bP}W(\bF,\bP) \sim \epsi^2 \partial_{\dot P}R_0(n,\dot
p)-\epsi n = \epsi \partial_{\dot p}R_0(n,\dot
p)-\epsi n.$$

In conclusion, the linearization the finite-strain constitutive model reads
\begin{align*}
  &\sigma = \mathbb C(\eta^s - p^s), \quad n = \sigma -\mathbb H
    p,\\
  &\partial_{\dot p}R_0(n,\dot p) \ni n, \quad R_0(n,\dot p)
    =\sup\{\dot p:\haz n \,| \, g(I,\haz n) \leq r(I,n)\}
\end{align*}
which corresponds to the linearized model in \cite{Ulloa2021}.
Note however that the above linearization argument is just
formal and should be paired with the proof that (suitably rescaled)
finite-strain solutions trajectories converge to linearized ones as $\epsi \to 0$,
in the spirit of~\cite{ms5}.\

\subsection{Dissipation}\label{sec:dissipation}
To simplify notation in the flow rule, from here on we handle the
thermodynamic force $N$
from \eqref{eq:N}
as a function of $F$ and $P$, namely,
$$N=N(F,P) = \bP^{-\TT}
F^\TT \partial_{\bFe} 
\We(F P^{-1}) \bP^{-\TT}-\partial_{\bP}\Wh(\bP).$$
Consequently we consider the
infinitesimal dissipation as a
  function of the variables $(P,F,\dot P)$ by letting    
  $$R(F,P,\dot P) = R(P,N(F,P),\dot P)$$
  without introducing new notation.
 Recall that $W_e$ and $W_p$ are assumed to be smooth, so that $N$ turns out to be a smooth
function of $F$ and $P$, as well.

Motivated by the von Mises example of Section
 \ref{sec:vonMises}, in particular by bounds \eqref{eq:h}, in the following we will assume that
 \begin{align}
&r_1\haz R(P,\dot P)
\leq R(F,P,\dot P) \leq \frac{1}{r_1}\haz R(P,\dot P)\quad \forall (F,P,\dot P)\in \Rz^{3\times 3}\times \SL \times \Rz^{3\times 3}\label{eq:Rnondegenerate}
 \end{align}
  for some given $r_1>0$, where $\haz R: \SL\times \Rz^{3\times 3} \to
  [0,\infty]$ is defined as
\begin{align}
    \haz R(P, \dot P): =\left\{
    \begin{array}{ll}
     |
      \dot P \bP^{-1}| & \text{if} \ \ \tr(  \dot P \bP^{-1})=0\\[1mm]
      \infty&\text{otherwise}.
    \end{array}
  \right. \nonumber
\end{align}
Moreover, we assume $R$ to be 
locally Lipschitz continuous on its domain and, in particular, lower
semicontinuous with respect to all variables. In the concrete von
Mises case of
\eqref{eq:locadiss2} this would directly follow  from the smoothness of 
\begin{align}\nonumber
    \bN\bP^{\TT}=\bP^{-\TT}
\bF^\TT \partial_{\bFe} 
\We(\bF \bP^{-1}) -\partial_{\bP}\Wh(\bP)\bP^{\TT} = F_e^\top \partial_{\bFe} 
\We(F_e) -\partial_{\bP}\Wh(\bP)\bP^{\TT} 
\end{align}
as a function of $F$ and $P$.

Building on the infinitesimal dissipation $R$, we introduce a
corresponding Finsler
metric on $\SL$ which measures the distance between two plastic strains. For all
given $F\in \Rz^{3\times 3}$ we define the local dissipation
$D(F,P_1,P_2)$ between $P_1\in \SL$ and $P_2\in \SL$ via
\begin{align}
    D( F, P_1, P_2):=\min&\bigg\{\int_0^1  R(F, P(s), \dot
                           P(s))\,\d s\: | \: P\in
  W^{1,1}(0,1;\SL), \nonumber\\
  &\qquad \qquad P(0)=P_1,\ P(1)=P_2\bigg\}.\label{eq:D}
\end{align}
Owing to the lower semicontinuity and the nondegeneracy
\eqref{eq:Rnondegenerate} of $R$ it is a standard matter to check that
the minimization in the definition of $ D( F, P_1, P_2)$ admits a solution. Moreover, the bounds
\eqref{eq:Rnondegenerate} entail that
\begin{equation}
  r_1 \haz D(P_1,P_2) \leq D(F,P_1,P_2)\leq \frac{1}{r_1}\haz
D(P_1,P_2) \quad \forall (F,P,\dot P)\in \Rz^{3\times 3}\times \SL
\times \Rz^{3\times 3} \label{eq:boundsD}
\end{equation}
where $\haz D(P_1,P_2)$ is  defined by the analogous minimization
problem \eqref{eq:D}, based on the infinitesimal dissipation $\haz
R(P,\dot P)$. The dissipation $\haz D$ is well studied
\cite{Mielke2002,Mielke2003}. In particular, we have that
\begin{align}
  &\haz D(P_1,P_2) =0 \ \Leftrightarrow \ P_1=P_2,\label{eq:nonodeg}\\
  & \haz D(P_1,P_2) \leq c( 1 + |P_1| + |P_2|)\quad \forall P_1, \, P_2
    \in \SL,\label{eq:boundd}\\
    &(P_1,P_2) \mapsto \haz D(P_1,P_2)  \quad \text{is continuous}. \label{eq:boundd2}
\end{align}
These properties directly translate  to the
analogous properties for $D(F,\cdot)$, uniformly with respect to $F \in
\Rz^{3\times 3}$. Moreover, the very definition \eqref{eq:D} and the positive $1$-homogeneity of $R$
with respect to $\dot P$ yield the triangle inequality
\begin{equation}
  \label{eq:tria}
  D(F,P_1,P_3) \leq D(F,P_1,P_2) + D(F,P_2,P_3) \quad \forall F \in
  \Rz^{3\times 3}, \ P_1,\, P_2, \, P_3 \in \SL.
\end{equation}
Indeed, let $P_{12},\, P_{23}: [0,1] \to \SL$ be  optimal curves for \eqref{eq:D}
connecting $P_1$ to $P_2$ and $P_2$ to $P_3$, respectively. The curve
given by
$P_{13}(t) = P_{12}(2t)$ for $t \in [0,1/2)$ and $P_{13}(t) =
P_{23}(2t-1) $ for $t\in [1/2,1]$ is such that
\begin{align*}
  &\int_0^1 R(F,P_{13}(t), \dot P_{13}(t))\, \d t \\
  &\quad=  \int_0^{1/2}
  R(F,P_{12}(2t), 2\dot P_{12}(2t))\, \d t  +  \int_{1/2}^1
    R(F,P_{23}(2t-1), 2\dot P_{23}(2t-1))\, \d t  \\
  &\quad =  D(F,P_1,P_2) + D(F,P_2,P_3).
\end{align*}
Hence, the triangle inequality \eqref{eq:tria} follows by taking the minimum on the left-hand
side with respect to all curves connecting $P_1$ and $P_3$.

Still with reference to example \eqref{eq:locadiss2}, by assuming
$r$ to be Lipschitz continuous with respect to $F$, uniformly in $P$ we
have that
\begin{align}
  &|R(F_1,P,\dot P) - R(F_2,P,\dot P) | \leq c |F_1 - F_2|\haz
    R(P,\dot P)\nonumber\\
  &\qquad \forall F_1,\, F_2 \in \Rz^{3\times 3}, \ P \in \SL, \, \dot P
    \in \Rz^{3\times 3},\label{eq:boundd2}
\end{align}
so that the analogous bound extends to $D$ and $\haz D$.

\section{Quasistatic-evolution system and regularization}\label{sec:newq}

In this section, we combine the constitutive relations
  \eqref{eqn:piola}--\eqref{eq:N} and the flow rule \eqref{combine1}
  with the quasistatic equilibrium system. Moreover, we comment on
  some simplification and regularization, which are needed for the
  analysis. 

  Let $\Omega\subset \mathbb{R}^3$ denote the reference configuration of the
elastoplastic body, which is assumed to be nonempty, open,
connected, bounded, and with Lipschitz boundary $\partial
\Omega$.

At all times $t \in [0,T]$, the deformation of the body is indicated
by $y(\cdot,t):\Omega\to\mathbb{R}^3$. We assume that the body is
clamped at $\Gamma_D\subset \partial\Omega$
  which is taken to be open in the topology of $\partial\Omega$ with
  $\mathcal{H}^2(\Gamma_D)>0$. The elastoplastic body is also
subjected to  a time-dependent
  body force with density $b:\Omega \times (0,T) \to \Rz^3$. In addition, a time-dependent
  traction $\gamma:\Gamma_N \times (0,T) \to \Rz^3$ is exerted at
  $\Gamma_N\subset \partial\Omega$, which is assumed to be open in the topology of
  $\partial\Omega$ and such that $\Gamma_D \cap \Gamma_N  = \emptyset$ and
  $\overline{\Gamma_D} \cup \overline{\Gamma_N} = \partial \Omega$.
  
  The quasistatic evolution system reads
  \begin{align}
    -{\rm div} \big(\partial_{F_e} W_e (\nabla y P^{-1}) P^{-\top}
    \big) = b\quad &\text{in} \ \ \Omega \times (0,T),\label{eq:1}\\
    \big(\partial_{F_e} W_e (\nabla y P^{-1}) P^{-\top}
    \big)n  = \gamma  \quad&\text{on} \ \  \Gamma_N \times (0,T), \label{eq:2}\\
    y={\rm id} \quad &\text{on} \ \  \Gamma_D \times (0,T),
                       \label{eq:3}\\
    (y(\cdot,0),P(\cdot,0)) = (y_0,P_0) \quad &\text{on} \ \  \Omega, \label{eq:4}\\ 
    \partial_{\dot P} R(\nabla y,P,\dot P) + \partial_P W(\nabla y,P)\ni 0 \quad &\text{in} \ \
                                                      \Omega \times
                                                      (0,T),\label{eq:5}
  \end{align}
  where in \eqref{eq:2} we indicate by $n$ the outward pointing normal
  to $\partial \Omega$ and $y_0:\Omega \to \Rz^3$ and $P_0:\Omega \to
  \SL$ in \eqref{eq:4} are given initial conditions.

Given the strong nonlinearities involved, proving the strong
solvability of problem \eqref{eq:1}--\eqref{eq:5} seems  
out of reach. Before proceeding to a variational reformulation of the
problem in
Section \ref{sec:energy}, we discuss here
some necessary simplification and regularization of the problem, which will eventually allow
for a mathematical treatment in Theorem \ref{thm:existencemv} below.

As a first regularization, we will include in the energy an additional term
penalizing space variations of the plastic strain. The {\it total
  stored energy} of the elastoplastic body will then read 
 $$\int_\Omega \We( \nabla y
\, \bP^{-1}) + \Wh(\bP)\, \d x + \frac{\mu}{ q_ r }\int_\Omega |\nabla \bP |^{ q_ r }
\, \d x,$$
where $ q_ r >3$ and $\mu>0$ are given.  In particular, $\mu$ is an additional length scale,
modulating the relative effect of the gradient
regularization. The plastic strain $P$ is hence assumed to be
in $W^{1, q_ r }(\Omega;\SL)$, which in particular entails that it is
continuous. 

In the elastoplastic context, penalizations of the gradient of $P$ in the
energy are classical and go under the name of {\it
  gradient-plasticity} theories  \cite{Fleck,Fleck2,Aifantis}. This
gradient term models nonlocal effects caused by short-range
interactions among dislocations \cite{Dillon,Gurtin,Gurtin2}. Apart
from the few exceptions \cite{compos2,compos}, all existence results
for finite-strain evolutive
elastoplasticity rely on gradient theories, see
\cite{Davoli,cplas2,Grandi2,Mielke16,Rindler2,Roeger} among others. In particular, the
incremental result of \cite{Mielke2005} features a penalization of the gradient
term ${\rm curl}\, P$ whereas the 
reference quasistatic evolution result \cite{Mainik} the full $H^1$-norm
of $P$ is penalized instead.

Note that the occurrence
of such a gradient term does not affect the dissipativity nature of the
model. The computation leading to the local dissipation inequality
\eqref{Clausius1} is still valid, up to considering the total stored
energy above and assuming   Neumann boundary conditions on $P$. More precisely, by assuming smoothness one has that
\begin{align}
  &\frac{\d}{\d t} \int_\Omega \left(W(\bF,\bP)
    +\frac{\mu}{ q_ r }|\nabla P|^{ q_ r }-  \bPi :
    \dot F \right)\, \d x\\
  &\quad= - \int_\Omega N : \dot P \, \d x- \mu \int_\Omega |\nabla
    P|^{ q_ r -2} \nabla P \vdots \nabla \dot P \, \d x   = -\int_\Omega R(F,\bP,\dot P) \, \d x \leq 0
\label{Clausius2}
\end{align}
where we have used the corresponding extended flow rule
$$\partial_{\dot P} R(\nabla y,P,\dot P) + N - \mu\, {\rm div}  \left(
|\nabla P|^{ q_ r -2}\nabla P \right) \ni 0 \quad \text{in} \ \
\Omega \times   (0,T),$$
see \eqref{eq:5}.
Note nonetheless that  \UUU \eqref{Clausius2} \EEE will be ascertained
below \UUU in a
weaker form, \EEE due to the limited regularity of the solution. By
including the gradient term in the total energy, the variable $P$
gains compactness, which helps in the analysis. 

In the following, we will assume the elastic energy density $W_e$ to
be defined on the whole $\Rz^{3\times 3}$ and to be of polynomial
growth. This simplification seems unavoidable  in order to be able
to mathematically tackle 
the problem. On the other hand, it is restrictive from
mechanical viewpoint, as we will not be able to guarantee that  
$\det F_e>0$ along the evolution. Let us however mention that such a
polynomial growth assumption is well motivated if $F_e$
stays in a small neighborhood of the identity. Differently from the
elastic case, this is not
unexpected  in the frame of elastoplasticity, where large deformations
generate large plastic strains, as opposed to large elastic
strains. As such, we believe this simplification to be acceptable in
the present setting  also from the mechanical viewpoint.

The assumed local Lipschitz continuity of the infinitesimal
dissipation $R$ from Section \ref{sec:dissipation} is unfortunately not enough for proving an existence
result. Besides the smoothness of the state dependence of $R$ on the
state $(F,P)$, one needs that the problem guarantees strong
compactness in the variables $(F,P)$. Such strong compactness is available for the variable $P$, as
effect of the above-introduced gradient term in the energy. On the
other hand, $F$ turns out to be just weakly compact on each
energy sublevel. In order to overcome this obstruction, we modify the model by
letting $F$ influence $R$ via a nonlocal, space-time convolution
term. Namely, we redefine 
$$ R = R(K F,P, \dot P)$$
where the nonlocal operator $K$ is given by
$$KF(x,t) = \int_0^t \int_{\Rz^3} \kappa(t-s) \phi(x-z) F(z,s) \, \d x
\, \d s \quad \forall (x,t) \in \Omega \times (0,T).$$
Here, $\kappa:\Rz\to [0,\infty)$ and $\phi: \Rz^3 \to [0,\infty)$ are  
smooth and $F$ is assumed to be trivially extended out of $\Omega$,
without introducing new notation, see \eqref{eq:KK} below.

Note that the above-mentioned compactness obstruction is present in the
linearized nonassociative setting of \cite{Mora}, as well, and a
space-time convolution of the stress is considered there,
too 
\cite{Mora2}. Besides nonassociative plasticity, analogous
regularizations via
mollification have been instrumental in the analysis of different
elastoplastic models, including Cam-Clay plasticity
\cite{DalMaso1,DalMaso2,DalMaso3} and Armstrong-Frederick nonlinear
kinematic hardening \cite{Francfort}. To date, we are unaware of
existence results for nonassociative plasticity without
mollifications, even in the linearised setting.   

From the modeling viewpoint, the above nonlocal modification
guarantees that the thermodynamic force $N$ driving the plasticization is
continuous. The occurrence of the term $KF$ in $R$
would correspond to a nonlocal dependence of the yield function $f$
and the plastic potential $g$ on $F$. Note that the convolution $K$ is {\it
  causal}, as the values of $KF(t)$ dependent on $F(s)$ for $s\in
(0,t)$ only. In addition, we do not prescribe the
size of the supports of $\kappa$ and $\phi$, so that in practice the effect of the
nonlocal term can be arbitrarily localized.

Before closing this section, let us mention that another possibility to
guarantee the strong compactness of $F$ in space and time would be that of including
additional gradient terms $\nabla F$ in the energy. This would
however require prescribing boundary conditions for $F$, an option
which might limit applications. In addition,
one should still consider viscous effects in time. Here, we prefer to resort to
the above regularization approach via mollification  which we
believe  \UUU to be  conceptually closer \EEE  to the
original model, as well as more transparent in its aims.

    
\section{Variational formulation and existence}\label{sec:energy}

In this section, we present the variational formulation of the
quasistatic evolution problem, introduce the concept of measure-valued solution, and
state our main existence result, namely, Theorem \ref{thm:existencemv}.

\subsection{Assumptions}\label{sec:assumptions} Let us start by introducing the functional
setting and specifying our assumptions. These will be tacitly assumed throughout
the paper.

Recall that the reference configuration
$\Omega\subset\mathbb{R}^3$ is nonempty, open, connected, bounded, and
Lipschitz and
that $\Gamma_D,\, \Gamma_N\subset \partial\Omega$ are open in the topology of
$\partial\Omega$ with $\mathcal{H}^2(\Gamma_D)>0$, $\Gamma_D \cap
\Gamma_N=\emptyset$, and $\overline{\Gamma_D} \cup \overline{\Gamma_N}
=\partial \Omega$.

Let $ q_ r >3$ and
$q,\,q_e,\,q_p>1$ fulfill
\begin{align}
    \frac{1}{q_e}+\frac{1}{q_p}\leq \frac{1}{q}<\frac{1}{3}. \label{esp1}
\end{align} The space of admissible states  is given as
 $\mathcal{Q}:=\mathcal{Y}\times\mathcal{P}$ where
\begin{align*}
  \mathcal{Y}&:=
               \{y\in W^{1, q }(\Omega;\mathbb{R}^3):\,y={\rm id}\,\,\,\text{on}\,\,\,\Gamma_D\},\\
    \mathcal{P}&:=W^{1, q_ r }(\Omega;\SL).
\end{align*}
We endow the space $\mathcal{Q}$ with the weak topology of the product
$W^{1,q}(\Omega;\mathbb{R}^3)\times
W^{1, q_ r }(\Omega;\mathbb{R}^{3\times 3})$.

Following the discussion of Sections \ref{sec:const}-\ref{sec:newq}, we assume the energy densities $W_e:\Rz^{3\times 3} \to
[0,\infty)$ and  $W_p:\mathbb{R}^{3\times 3} \to [0,\infty)$ to be smooth. Additional
dependencies on the material point $x\in \Omega$ are currently
excluded, but could be considered
with no particular intricacy. We ask
for constants  $c_1,\,c_2>0$ such that
\begin{align}\label{eq:coercivity1}
   -\frac{1}{c_1}+c_1|F_e|^{q_e}&\leq W_{e}(F_e) \leq      
                                  \frac{1}{c_1}(1 +|F_e|^{q_e})  \quad
                                                     \forall F_e\in
                                                     \Rz^{3\times 3},\\\label{eq:coercivity2}
   -\frac{1}{c_2}+c_2|P|^{q_p}&\UUU \leq W_{p}(P)\EEE \,\,\,\,\,
                                 \forall P \in \mathbb{R}^{3\times 3}.
\end{align}
Moreover, we  assume that the  {\it recession functions}
\begin{equation}\label{eq:rec}
  W^\infty_e(F_e): = \lim_{s\to \infty}\frac{W_e(sF_e)}{s^{q_e}} 
  \quad \text{and} \quad
  W^\infty_p(P): = \lim_{s\to \infty}\frac{W_p(sP)}{s^{q_p}}
\end{equation}
 are well-defined for all  $F_e,\,P\in \Rz^{3\times 3}$ 
 and are  continuous.

We assume that the elastic energy density $W_e$ is {\it polyconvex},
namely that there exists a convex and lower semicontinuous function $\mathbb{W}:\mathbb{R}^{3\times 3}\times \mathbb{R}^{3\times 3}\times \mathbb{R}_+\to [0,\infty)$ such that
\begin{align}\label{eq:polyconvexity}
    W_{e}(F_e)=\mathbb{W}(F_e,{\rm cof} F_e,{\rm det} F_e)\quad
  \forall F_e\in  \Rz^{3\times 3}.
\end{align}
Note that the assumptions on $W_e$ are compatible with {\it frame
  indifference}, namely,   $W_e(Q
F_e)=W_e(F_e)$ for all $F_e\in \Rz^{3\times 3}$ and $Q\in \SO$.

According to  Section \ref{sec:newq}, the  {\it stored energy} $\mathcal{W}:\,\mathcal{Q}\to [0,\infty]$ is
defined as the integral of the sum of the local and nonlocal energy
densities, namely,
\begin{align*}
  \mathcal{W}(y,P)&:=\int_\Omega W(\nabla y P^{-1},P,\nabla P)\,\d x\\
  &:=
  \int_\Omega W_{e}(\nabla y P^{-1})\,\d x+\int_\Omega W_p(P)\,\d x+\frac{\mu}{ q_ r }\int_\Omega|\nabla P|^{ q_ r }\,\d x,
\end{align*}
where $\mu>0$ is given.

Given the body-force and the surface-traction densities
$b$ and $\gamma$, the {\it generalized work of external forces} is defined as
\begin{align*}
    \la l(t),y\ra:=\int_\Omega
  b(x,t)\cdot y(x)\,\d x+\int_{\UUU \Gamma_{\text{N}}}\gamma(x,t)\cdot y(x)\,\d\mathcal{H}^2(x)
\end{align*}
where $\la\cdot,\cdot\,\ra$ indicates the duality pairing between
$\mathcal{Y}^\ast$ and $\mathcal{Y}$. We assume that 
\begin{equation}
  l\in C^{1}([0,T];\mathcal{Y}^\ast).\label{eq:elle}
\end{equation}
The latter immediately follows if $b \in
C^{1}([0,T];L^1(\Omega;\Rz^3))$ and $\gamma\in  
C^{1}([0,T];L^1(\Gamma_N;\Rz^3))$.

  The {\it
  total energy} $\mathcal{E}:[0,T]\times \mathcal{Q} \to
\Rz$ of the system  is hence given by
\begin{align*}
    \mathcal{E}(t,y,P):=\mathcal{W}(y,P)- \la l(t),y\ra.
\end{align*}

Following the discussion of Section \ref{sec:dissipation},
specifically relations \eqref{eq:boundsD}--\eqref{eq:boundd2}, we
assume the {\it local dissipation} $D=D(F,P_1,P_2): \Rz^{3\times
  3} \times \SL \times \SL \to [0,\infty)$ to be continuous and such that there
exist $\haz D:\SL\times \SL \to [0,\infty]$ and $c_3>0$ with
\begin{align}
  &   D( F,P_1,P_2)=D(F,P_2,P_1) \quad
    \forall  F \in \Rz^{3\times 3}, \ P_1,\,P_2 \in \SL, \label{eq:D0}\\
  &   D( F,P_1,P_2)=0 \ \Leftrightarrow
    \ P_1=P_2 \quad \forall F \in \Rz^{3\times 3},\label{eq:D1}\\
  & \haz D(P_1,P_2)\leq c_3D( F,P_1,P_2)\quad
    \forall F \in \Rz^{3\times 3}, \ P_1,\,P_2 \in \SL, \label{eq:D11}\\
  & D( F,P_1,P_3) \leq D( F,P_1,P_2)
    +D( F,P_1,P_2) \quad \forall F \in \Rz^{3\times 3}, \  P_1,\,
    P_2,\, P_3 \in \SL, \label{eq:D2}\\
   & D( F, P_1, P_2)\leq c_3 (1+|P_1|+|P_2|)\quad \forall F \in \Rz^{3\times 3}, \
    P_1,\, P_2, \in \SL, \label{eq:D3}\\
  &|D(F_1,P_1,P_2) - D(F_2, P_1,P_2) |\leq c_3 |F_1 - F_2| \, \haz D(P_1, P_2) \nonumber\\
  &\quad \forall F_1,\, F_2 \in \Rz^{3\times 3}, \ P_1,\, P_2, \in \SL.\label{eq:D4}
\end{align} 

For all given $\haz F\in
L^\infty(\Omega;\Rz^{3\times 3})$, by taking the integral  over $\Omega$ we define the {\it
  dissipation} $\mathcal{D} (\haz F,P_1,P_2)$ between   $P_1\in
L^\infty(\Omega;\SL)$ and $P_2\in  L^\infty(\Omega;\SL)$ as
\begin{align*}
    \mathcal{D}( \haz F, P_1, P_2):=\int_\Omega D( \haz F(x),P_1(x), P_2(x))\,\d x.
\end{align*}

Eventually, for all trajectories $ (\haz F,P) :[0,T]\to
L^\infty(\Omega;\mathbb{R}^{3\times 3}\times \SL)$  we define the {\it
  total
dissipation} on the time interval $[s,t] \subset [0,T]$  as
\begin{align*}
    {\rm Diss}_{[s,t]}( \haz F, P):=\sup \left\{\sum_{i=1}^m \mathcal{D}\big(\haz F(t_{i-1}), P(t_{i-1}), P(t_{i})\big):s{=}t_0{<}t_1{<}...{<}t_m{=}t \right\}
\end{align*}
where the supremum is taken over all partitions of $[s,t]$.

Given the kernels $\kappa\in W^{1,1}(0,T)$ and  $\phi\in
W^{1,\infty}(\mathbb{R}^3)$, we define 
  the space-time convolution operator
$K:L^\infty(0,T;L^q(\Omega;\Rz^{3\times 3})) \to L^\infty(\Omega
\times (0,T);\Rz^{3\times 3})$ as
\begin{align}
K  F(x,t):=(\kappa\ast(\phi\star  F))(x,t) \label{eq:KK}
\end{align}
where the symbols
$\ast$ and $\star$ respectively denote the standard convolution
products on $(0,t)$ and $\mathbb{R}^3$. Namely, for all $(x,t) \in
\Omega \times (0,T)$ we let
\begin{align*}
    (\kappa\ast  F)(x,t):=\int_0^t\kappa(t-s)  F (x,s)\,\d s \ \ \text{and} \ \
    (\phi\star  F)(x,t):=\int_{\mathbb{R}^3}\phi(x-z)  F (z,t)\,\d z.
\end{align*}
In the latter,  the trivial extension of $ F (\cdot,t)$ to the whole $\Rz^3$ is
considered, without introducing new notation.

 Note that
we define $\mathcal{D}$ for $\haz F: [s,t]\to
L^\infty(\Omega;\mathbb{R}^{3\times 3})$ only, as $\haz F$ is actually a
place holder for $K \nabla y$. In fact, owing to the coercivity from
Lemma \ref{lemma:comp} below, we will have that  
$\nabla y \in L^\infty(0,T;L^q(\Omega;\Rz^{3\times 3})) $, so that   
$K\nabla y$ is bounded almost everywhere in $\Omega \times (0,T)$.

Eventually, we ask to be given the initial state $P_0\in \mathcal{P}$
such that there exists a $y_0 \in \mathcal{Y}$ such that  
   \begin{align}
     \mathcal{E}(0,y_0,P_0)\leq\mathcal{E}(0,\hat{y},\hat{P})+
     \mathcal{D}\big( 0,P_{0},\hat P\big),\,\,\,\text{for all}\,\,(\hat y,\hat P)\in\mathcal{Q}.\label{eq:stab0}
   \end{align}
 Namely, we assume that the initial state is {\it stable}, see
 \eqref{def:stability} below.
 
\subsection{\UUU Generalized Young measures} \UUU
 We
start by introducing \EEE some notation and recalling some
basic fact,  referring, for instance, to \cite{Fonseca,Bogdan2020,Kristensen} for additional
detail.

Let $h : \mathbb{R}^m\to\mathbb{R}$ be a continuous function such
that the corresponding  {\it $p$-recession function} $ h^\infty
: \mathbb{S}^{m-1}\to\mathbb{R}$ given by
\begin{align*}
    h^\infty(\tilde z):=\lim_{s\to\infty}\frac{h(s\tilde
  z)}{s^p} 
\end{align*}
is well-defined and continuous. 
\UUU Henceforth, the recession function $h^\infty$ is considered to be
extended to a $p$-homogeneous function from $\mathbb{R}^m$ to
$\mathbb{R}$ without changing notation. \EEE

Given $E\subset \Rz^m$ measurable, we indicate by $\mathcal{M}(E)$ and
$\mathcal{M}^+(E)$ the Radon measures and the positive
Radon measures on $E$, respectively. Given
a weakly convergent sequences
$(v_n)\subset L^p(\Omega,\mathbb{R}^m)$ one can find
\cite{Alibert,DiPerna} a not relabeled
subsequence and a triplet of measures $(\nu_x,\lambda,\nu^\infty_x)$
such that  
\begin{align}
  &\int_{\Omega}h(v_n(x))\varphi(x)\,\d x\nonumber\\
  &\quad \to\int_{\Omega}\int_{\mathbb{R}^m}h(z)\,\d{\nu}_x(z)\,\varphi(x)\,\d x
		+ \int_{\overline{\Omega}}\int_{\mathbb{S}^{m-1}}h^{\infty}(\tilde
  z)\:\,\d{\nu}_x^{\infty}(\tilde z) \,\varphi(x) \,\d{\lambda}(x), \label{eq:Young}
\end{align}
for all $\varphi\in C(\overline{\Omega})$ and for any such functions
$h$. Here, the  {\it oscillation measure} $\nu=(\nu_x)_{x\in \Omega }\in
   L^\infty_{w^\ast}(\Omega;\mathcal{M}(\mathbb{R}^m))$ is a
   parametrized Radon measure depending  weakly$^*$ measurably on the
parameter $x\in \Omega$ with respect to the Lebesgue measure, $\lambda \in
   \mathcal{M}^+(\overline{\Omega})$ is the nonnegative {\it concentration
       measure}, and $\nu^\infty=(\nu_x^\infty)_{x\in\overline \Omega}\in
 L^\infty_{w^\ast}(\overline \Omega
 ,\lambda;\mathcal{M}(\mathbb{S}^{m-1}))$ is the parametrized probability {\it concentration-angle
         measure} depending  weakly$^*$ measurably on the
   parameter $x\in \overline \Omega$ with respect to the measure
   $\lambda$. Note that  the measure $\nu^\infty_x$ is defined $\lambda$-almost
everywhere.
 In case \eqref{eq:Young} holds we say that   the {\it Young measure}
$\nu:=(\nu_x,\lambda,\nu_x^\infty)$ is {\it generated} by the weakly
convergent sequence $(v_{n})$.   

 In  the specific setting of our model, the above-recalled
classical setting needs to be slightly extended in order to deal with different homogeneities
across components. Precisely, we deal with Young measures 
generated by sequences $(F_n,P_n,G_n)$ which are weakly converging  in
$$ L^{q_e}(\Omega;\Rz^3)\times L^{q_p}(\Omega; \Rz^{3\times 3})\times
L^{ q_ r }(\Omega;\Rz^{3\times 3 \times 3})$$ and we are interested in the
limiting behavior of the
sequence $W (F_n,P_n,G_n)$ under assumptions \eqref{eq:coercivity1}--\eqref{eq:rec}. This can be dealt with by introducing the  {\it nonhomogeneous unit sphere}
\begin{align*}
    \mathbb{S}_{q_e,q_p, q_ r }:=\left\{\tilde z:= (\tilde z_e,\tilde
  z_p,\tilde z_r)  \in\mathbb{R}^{3\times 3}\times \mathbb{R}^{3\times 3}\times
  \mathbb{R}^{3 \times 3 \times 3}\: | \ |  \tilde z_e
  |^{q_e}+|  \tilde z_p |^{q_p}+|  \tilde z_r |^{ q_ r }=1\right\}.
\end{align*}
 The classical construction \cite{Alibert,DiPerna} can hence be
adapted to ensure that for any such
weakly convergent sequence one  may find  a not relabeled
subsequence $(F_n,P_n,G_n)$ and  a Young measure
$\nu=(\nu_x,\lambda,\nu^\infty_x)$ with  $\nu\in
L^\infty_{w^\ast}(\Omega;\mathcal{M}(\mathbb{R}^{m}))$,
$\lambda\in\mathcal{M}^+(\overline{\Omega})$,  and $\nu^\infty\in
L^\infty_{w^\ast}(\Omega,\lambda;\mathcal{M}(\mathbb{S}_{q_e,q_p,
  q_ r }))$  with $m=3 + (3\times 3) + (3\times 3\times 3)$ such that
\begin{align*}
     W(F_n(x),P_n(x),G_n(x))\:\d x\overset{\ast}{\rightharpoonup}\int_{\mathbb{R}^{m}}W(z)\,\d{\nu}_x(z)\,\d x
		+ \int_{\mathbb{S}_{q_e,q_p, q_ r }}W^{\infty}(\tilde z)\,\d{\nu}_x^{\infty}(\tilde z) \,\d{\lambda},
\end{align*}
weakly$^\ast$ in the sense of measures,  see \eqref{eq:Young}. Here, the {\it nonhomogeneous
recession function} $ W^\infty: \mathbb{S}_{q_e,q_p, q_ r } \to \Rz$ is
defined as
\begin{align*}
    W^\infty(\tilde z):=\lim_{s\to\infty}
		\frac{W(s^{q_p  q_ r } \tilde z_e 
  ,s^{q_e  q_ r } \tilde z_p  ,s^{q_e q_p} \tilde
  z_r  )}{s^{q_eq_p q_ r }} =  W^\infty_e(\tilde z_e )
   +
  W^\infty_p( \tilde z_p  ) + \frac{\mu}{ q_ r }|
  \tilde z_r  |^{ q_ r }
\end{align*}
where the last equality comes from \eqref{eq:rec}.

 A further extension is then required by the fact that we deal with the time-dependent,
quasistatic setting. The generating sequences $(F_n,P_n,G_n)$ \UUU additionally depend
on time and  weakly$^\ast$ converge \EEE in 
$$ L^\infty(0,T;L^{q_e}(\Omega;\Rz^3)\times L^\infty(0,T;L^{q_p}(\Omega;\Rz^{3\times 3})\times
L^\infty(0,T;L^{ q_ r }(\Omega;\Rz^{3\times 3 \times 3})).$$ 
Correspondingly, the
generated Young measure $(\nu_{x,t},\lambda,\nu^\infty_{x,t})$ is additionally parametrized by time. We will use the following structural result.  

\begin{Lemma}[Structural  result,
  \cite{Brenier2011}]\label{lemma: spartYM}
  Let the sequence  $(F_n,P_n,G_n)_n$ be bounded in $$L^\infty(0,T;L^{q_e}(\Omega;\mathbb{R}^{3\times 3})\times
L^{q_p}(\Omega;\mathbb{R}^{3\times 3})\times
L^{ q_ r }(\Omega;\mathbb{R}^{3\times 3\times 3}))$$ and
generate the Young measure 
$\nu=\left(\nu_{x,t},\lambda,\nu_{x,t}^\infty\right)$ with  $\nu_{x,t}\in
L^\infty_{w\ast}(\Omega \times (0,T);\mathcal M(\Rz^m))$, $\lambda \in
 \mathcal M^+(\overline \Omega \times [0,T])$, and $\nu_{x,t}^\infty\in
L^\infty_{w\ast}(\Omega \times (0,T);\mathcal M(\mathbb
S_{q_e,q_p,q_r}))$. Then, 
\begin{align*}
    \sup_t\left(\int_\Omega\la\nu_{x,t},| z_e |^{q_e}+| z_p |^{q_p}+| z_r |^{ q_ r }\ra \,\d x\right)<\infty,
\end{align*}
 and  the concentration measure $\lambda$ admits a disintegration of the form $\lambda(\d x,\d t)=\eta_t(\d x)\otimes \d t$, where $t\mapsto \eta_t$ is a bounded measurable map from $[0,T]$ into $\mathcal{M}^+(\overline{\Omega})$.
\end{Lemma}
 In case $q_e=q_p= q_ r=2 $ this fact is reported in
\cite{Brenier2011}. A detailed proof
can also be found in \cite[Prop.~3.6]{WiedT2012}, as well. The case of
general exponents can be proved analogously.

\subsection{Main result} \UUU We are now ready to introduce our notion
of solution \EEE based
on the concept of Young  measures.  We resort to a single
Young measure $\nu= (\nu_{x,t},\lambda,\nu_{x,t}^\infty)$ in order to
qualify the  weak
limit of sequences $$(\nabla y_nP_n^{-1},P_n, \nabla P_n)\in
L^\infty(0,T; L^{q_e}(\Omega;\Rz^{3\times  3  })\times L^{q_p}(\Omega;\SL)\times
L^{ q_ r }(\Omega;\Rz^{3\times 3\times 3})).$$ 
The action of the Young measure $\nu$ on the energy density $W$ is
denoted by 
\begin{align}
     \la\!\la\nu_t,W\ra\!\ra:= \int_\Omega \la\nu_{x,t},\tilde W(
  z_e , 
  z_r )\ra  + W_p(P(t))\,\d x+\int_{\overline{\Omega}} \la\nu_{x,t}^\infty,\tilde W^\infty(
  z_e , 
  z_r )\ra \,\d\eta_{t}(x),\label{eq:action}
\end{align}
where
 
\begin{align*}
    \tilde W(\nabla y P^{-1},\nabla P)=W_{e}(\nabla y P^{-1})+\frac{\mu}{ q_ r }|\nabla P|^{ q_ r },
\end{align*}
represent the {\it concentration part} of the energy $W$ and $\tilde
W^\infty( z_e, z_r )= W^\infty_e(z_e)  +\mu| z_r |^{ q_ r }/ q_ r $
is the corresponding recession function.

\begin{Definition}[Energetic measure-valued solution]\label{def:ensol}\sl
  An \emph{energetic
    measure-valued solution}  is a triple $(y,P,\nu)$ consisting
  of a function $y\in
  L^\infty(0,T;\mathcal{Y})$, a mapping $P:[0,T]\to
  \mathcal{P}$,  and a  Young measure
  $\nu:=\left(\nu_{x,t},\lambda,\nu_{x,t}^\infty\right)$ such that
\begin{align*}
        \la\nu_{x,t}, z_e \ra&= \nabla y P^{-1}    \quad\text{in\:} L^{\infty}(0,T;L^{q_e}(\Omega)) ,\\
    \la\nu_{x,t},
                     z_p  \ra &= P  \hspace{1.2cm}\text{ in\:} L^{\infty}(0,T;L^{q_p}(\Omega)) ,\\
        \la\nu_{x,t}, z_r \ra &= \nabla P \hspace{0.95cm}\text{in\:} L^{\infty}(0,T;L^{ q_ r }(\Omega)) ,
	\end{align*}
and for a.e. $t\in (0,T)$ the following 
conditions hold: 
\begin{align}
  &\text{\rm Stability:}\quad\nonumber\\[1mm]
  &\la\!\la\nu_{ t},W\ra\!\ra-\langle l(t),y(t)
    \rangle \nonumber\\
  &\qquad \leq
  \int_\Omega W(\nabla \hat y\hat {P}^{-1},\hat P,\nabla \hat
  P)\,\d x-\langle l(t),\hat y\rangle + {\mathcal D} ( K\nabla y(t),
    P(t), \hat P)   \quad 
    \forall (\hat y, \hat P)\in \mathcal{Q} ,\label{def:stability}\\[5mm]
  &\text{\rm Energy balance:} \nonumber\\[1mm]
  &\la\!\la\nu_t,W\ra\!\ra-\langle l(t),y(t)\rangle+{\rm
                                       Diss}_{[0,t]}(K\nabla y,P) \nonumber\\
  &\quad = \int_\Omega W(\nabla y_0  {P}^{-1}_0,P_0,\nabla P_0)\,\d
    x-\langle l(0),y_0\rangle -\int_0^t
    \la \dot l(r),y(r)\ra \,\d r .\label{def:energy_balance}
\end{align}
\end{Definition}

Our main result reads as follows.

\begin{Theorem}[Existence]\label{thm:existencemv}
  Under the assumptions of Section \emph{\ref{sec:assumptions}}, there exists a measure-valued energetic solution $(y,P,\nu)$ with $P(0)=P_0$.   
\end{Theorem}

Theorem \ref{thm:existencemv} is proved in Section \ref{sec:proof}
below. 

\subsection{A formulation in terms of functions}
Let us comment now on the possibility of formulating an alternative existence
statement purely in terms of functions. This is indeed possible, \UUU at
the expense of possibly considering \EEE two distinct limits for the discrete
deformations \UUU and obtaining an upper energy bound only. \EEE
In particular, we can prove the following.

\begin{Proposition}[Solutions as functions]
 \label{lemma:functions}
    Under the assumptions  of Section \emph{\ref{sec:assumptions}}
    there exists  an \emph{energetic function-valued solution}, namely a
    triple $(y,P,\tilde y)$ where $y\in
    L^\infty(0,T;\mathcal{Y})$ and  $(\tilde y,P): [0,T]\mapsto\mathcal{Q}$ with
    $(\tilde y(0), P(0))=(y_0,P_0)$ such that for a.e. $t\in (0,T)$
\begin{align}
  &\text{\rm Stability:}\nonumber\\
  &\int_\Omega W(\nabla \tilde y(t) P^{-1}(t), P(t),\nabla
  P(t))\,\d x-\langle l(t),\tilde y(t) \rangle \nonumber\\
  &\quad \leq \int_\Omega W(\nabla \hat y\hat {P}^{-1},\hat P,\nabla
    \hat P)\,\d x-\langle l(t),\hat y\rangle  + {\mathcal D}
    (K\nabla y(t),  P(t),\hat P) \quad  \forall (\hat y,\hat
    P)\in \mathcal{Q},  \label{eq:stab1}\\[2mm]
  &\text{\UUU \rm Upper energy estimate:\EEE }\nonumber\\[1mm]
  & \int_\Omega W(\nabla \tilde y(t) P^{-1}(t), P(t),\nabla
                                  P(t))\,\d x-\langle l(t),\tilde y(t)\rangle +{\rm
  Diss}_{[s,t]}\big(K\nabla y, P \big) \nonumber\\
  &\quad \UUU \leq \EEE \int_\Omega W(\nabla y_0  {P}^{-1}_0,P_0,\nabla P_0)\,\d
    x-\langle l(0),y_0\rangle -\int_0^t
    \la \dot l(r),y(r)\ra \,\d r.\label{eq:en1}
\end{align}
\UUU Moreover, $(y,P,\tilde y)$ are limits of discrete solutions $(\bar
y_n,\bar P_n)$ to \eqref{eq: dmin} as $\tau_n \to 0$. More precisely,
up to not relabeled subsequences
we have $\bar y_n\stackrel{\ast}{\rightharpoonup} y$ in $L^\infty(0,T;\mathcal{Y})$, $\bar
P_n(t)\to P(t)$ in $L^\infty(\Omega ; \Rz^{3\times 3})$ for all
$t\in [0,T]$, and $\ove y_{n^t_k}(t) \weak \tilde y(t)$ for all $t\in
[0,T]$ and for some subsequence $n^t_k\to \infty$ possibly depending on
$t$.

In case $y=\tilde y$ a.e., for a.e. $t \in (0,T)$ one has the 
\begin{align}
&\text{\rm  Energy balance: }\nonumber\\[1mm]
  & \int_\Omega W(\nabla \tilde y(t) P^{-1}(t), P(t),\nabla
                                  P(t))\,\d x-\langle l(t),\tilde y(t)\rangle +{\rm
  Diss}_{[s,t]}\big(K\nabla y, P \big) \nonumber\\
  &\quad \UUU =  \int_\Omega W(\nabla y_0  {P}^{-1}_0,P_0,\nabla P_0)\,\d
    x-\langle l(0),y_0\rangle -\int_0^t
    \la \dot l(r),\tilde y(r)\ra \,\d r.\label{eq:en100}
    \end{align}\EEE

\end{Proposition}

Proposition \ref{lemma:functions} is proved in Section
\ref{proof:lemma:functions}. The crucial point in the above statement
is that one is forced to use the two, possibly different functions
$\tilde y:[0,T]\to \mathcal{Y}$ and $y\in
L^\infty(0,T;\mathcal{Y})$. Indeed, $\tilde y$ and $y$ result from
limit procedures applied to some given piecewise constant
approximation $\UUU \bar y_n \EEE:[0,T]\to \mathcal{Y}$ with \UUU $\|
\bar y_n\|_{L^\infty(0,T;\mathcal{Y})}\leq c$, \EEE see Section
\ref{sec:discrete_solutions}. Owing to the latter bound, by passing to
some not relabeled subsequence one defines $y$ as the weak$^\ast$ limit
in $L^\infty(0,T;\mathcal{Y})$ of $\UUU \bar y_n$. \EEE Note that $y$ is a.e. defined
in $(0,T)$ and is measurable. On the other hand, for all
$t\in [0,T]$ one has that $(y_n(t))$ is bounded in $\mathcal{Y}$, so
that, by possibly extracting a $t$-dependent subsequence one finds a
pointwise limit $\tilde y(t)$. The function $\tilde y$ is defined
everywhere, possibly being not measurable.

These two distinct limiting procedures are instrumental for passing to
the limit in the discrete setting. In particular, the limit passage in
the dissipation term involves $y$, whereas the pointwise in time limit in
the energy term calls for $\tilde y$.
As pointwise limits need not be unique,
the two functions $y$ and $\tilde y$ do not a.e. coincide. \UUU This
prevents from proving the energy balance. Nonetheless, the upper energy
estimate \eqref{eq:en1} holds true. In addition, the {energetic
  function-valued solution} is {\it approximable}, namely, is the
limit of discrete problem.

In case $y=\tilde y$ concide almost everywhere, one can
obtain the energy balance \eqref{eq:en100}, as well. This is for instance the
case for a \EEE
convex energy density $W$. Note that such convexity cannot be
reconciled with frame indifference but may still be of practical
use.  In fact, 
 in the elastoplastic
setting the elastic strain $\nabla y P^{-1}$
stays close to identity anyways. If $W$ is convex, due to the
positivity of the concentration part of our measure  $\nu$,  we may apply Jensen's inequality to infer that
    \begin{align*}
       &\int_\Omega W(\nabla  y(t){P}^{-1}(t), P(t),\nabla
         P(t))\,\d x-\langle l(t),  y(t)\rangle\\
      &\quad\leq \int_\Omega \langle\nu_{x,t}\tilde
        W( z_{e}, z_{ q_ r })\rangle + W_p(P(t)) \,\d x-\langle
        l(t),  y(t)\leq\la\!\la\nu_{t},W\ra\!\ra-\langle l(t),y(t)
        \rangle \\
      &\quad\leq \int_\Omega W(\nabla \tilde y(t){P}^{-1}(t), P(t),\nabla P(t))\,\d x-\langle l(t), \tilde y(t)\rangle,
    \end{align*}
   for a.e. $t\in (0,T)$. Testing now \eqref{eq:stab1} with $\hat
   y=y(t)$ and $\hat P= P(t)$ we infer that the above inequalities
   hold as equalities and so energies of the pointwise limits, the weak$\ast$
   limit, and the limiting measure coincide. In other words, one can state
   \eqref{eq:stab1}--\eqref{eq:en1} in terms of a single function only,
   either $y$ or $\tilde y$.

  Let us conclude with a comment on the relation between energetic
 measure-valued and energetic function-valued solutions. The two
 concepts are strictly connected, as they correspond to two
 different ways to pass to the limit in time-discrete
 approximations, called {\it discrete solutions} in the following, see the scheme \eqref{eq:minproof} below. In fact, within the class of solutions obtained as
 limits of discrete solutions, the value of the energy is independent of
 the solution concept. We record these facts in the following.

 \begin{Proposition}[Correspondence of
   solutions]\label{prop:correspondence}
   For any energetic
   measure-valued solution $(y,P,\nu)$ which is the limit of
    discrete solutions, see \eqref{eq:minproof}, there exists an energetic
   function-valued solution of the form $(y,P,\tilde y)$. Viceversa, for
   any energetic function-valued solution $(y,P,\tilde y)$ there exists an energetic
   measure-valued solution of the form $(y,P,\nu)$. In both cases, for a.e. $t\in [0,T]$ one
   has, 
\begin{align}
 \la\!\la\nu_{t},W\ra\!\ra-\langle l(t),  y(t) \rangle = \int_\Omega W(\nabla \tilde y(t){P}^{-1}(t), P(t),\nabla P(t))\,\d x-\langle l(t), \tilde y(t)\rangle.\label{eq:rela}
 \end{align}
\end{Proposition}

A proof of Proposition \ref{prop:correspondence} is given in Section
\ref{sec:correspondence}.

\section{Proof of Theorem \ref{thm:existencemv}: existence of
  measure-valued solutions}\label{sec:proof}

\subsection{Coercivity of the energy}
A first consequence of the assumptions in Section
\ref{sec:assumptions} is that the sublevels of the energy $\mathcal
E(t,\cdot)$ are weakly compact. More precisely, we have the following
compactness result, whose proof is in \cite{Damage2019}.

\begin{Lemma}[Coercivity of the energy]\label{lemma:comp} There exists
  $c_4$ such that  
\begin{align*}
  \| \nabla y P^{-1}\|_{q_e}^{q_e}+  \|\nabla
  y\|_q^q+\|P\|_{q_p}^{q_p}+\|\nabla
  P\|_{ q_ r }^{ q_ r }\leq c_4(1+\mathcal{E}(t,y,P)) \quad
  \forall (t,y,P)\in [0,T]\times \mathcal{Q}.
\end{align*}
In particular, all sequences $(t_n,y_n,P_n)\in  [0,T]\times\mathcal{Q}$ with
     $   \sup_n\mathcal{E}(t_n,y_n,P_n)<\infty$ admit a weakly
     convergent subsequence  in $[0,T]\times\mathcal{Q}$.  
   \end{Lemma}

By computing the power of external forces  we get
\begin{align*}
    |\partial_t\mathcal{E}(t,y,P)|&=|\langle\dot l(t),y\rangle|\leq
                                    \|\dot
                                    l(t)\|_{\mathcal{Y}^\ast}\|y\|_{\mathcal{Y}
                                    }
                                   \leq c
  \|\dot l(t)\|_{\mathcal{Y}^\ast}\|\nabla y\|_q\\
  & \leq c  \|\dot l(t)\|_{\mathcal{Y}^\ast}(1+\|\nabla y\|^q_q)\leq
    \|\dot l(t)\|_{\mathcal{Y}^\ast} \big( 1+c_4(1+\mathcal{E}(t,y,P))\big),
\end{align*}
where we have used the Korn
inequality and the coercivity of the energy from
Lemma \ref{lemma:comp}, as well. This in particular entail the growth
control on the power of external forces 
\begin{align}\label{eq:rateenergybnd}
    |\partial_t\mathcal{E}(t,y,P)|\leq c_5\big(
  1+\mathcal{E}(t,y,P)\big)  \quad \forall (t,y,P)\in [0,T]\times \mathcal{Q}
\end{align}
where $c_5 = (1+c_4)\|\dot l\|_{L^\infty(0,T;\mathcal{Y}^*)}$. Moving from
\eqref{eq:rateenergybnd},  the
Gronwall lemma gives
\begin{align}\label{eq:dGronwall}
    1+\mathcal{E}(t,y,P)\leq c_5
  \big(1+\mathcal{E}(s,y,P)\big)e^{c_5(t-s)}\quad \forall
  (y,P)\in\mathcal{Q}, \ 0\leq s\leq t\leq T.
\end{align}

\subsection{Discrete solutions} \label{sec:discrete_solutions} Assume to be given a uniform partition
of the time interval $[0,T]$
of time step $\tau=T/n$ $(n \in \Nz)$, that is
$\{0=t_0<t_1<...<t_{n-1}<t_n=T\}$ with $t_i=i\tau$. Given any vector $(w_i)_{i=0}^{n}$, we will denote by
$\overline{w}_n$ the corresponding piecewise constant left-continuous
interpolant on the partition, namely,
\begin{align*}
    &\overline{w}_n(t):=w_0 \quad \forall t \leq 0, \\
  &\overline{w}_n(t):=w_i \quad \forall
    t\in\left((i-1)\tau,i\tau\right], \ i=1,\dots,n,\\
    &\overline{w}_n(t):=w_n \quad \forall t\geq T.
\end{align*}
Given the initial condition $(y_0,P_0)\in\mathcal{Q}$
and setting $l_i=l(t_i)$ for $i=0,\dots,n$,
we term $(y_i,P_i)_{i=1}^n\in\mathcal{Y}^n\times\mathcal{P}^n$ a {\it
  discrete solution} if it incrementally solves the minimization problems
\begin{equation}\label{eq: dmin}
  (y_i,P_i)\in {\rm Argmin} \left\{\int_\Omega W(\nabla y
    P^{-1},P,\nabla P)\,\d x-\langle l_i,y \rangle+{\mathcal D}\left((K_\tau\nabla y)_{i-1},P_{i-1},P\right)\right\},
\end{equation}
for $i=1,..,n$.
In \eqref{eq: dmin}, we use the {\it discrete-in-time} version $K_\tau$ of
the operator $K$ defined as
\begin{align*}
    (K_\tau w)_{i-1}(x):= (\kappa\ast_\tau(\phi\star w))_{i-1}(x)
  \quad \text{for a.e.} \ x\in\Omega,
\end{align*}
where, for $\kappa_i:=\kappa(t_i)$, $i=0,\dots,n,$ we define the {\it
  discrete convolution} $(\kappa\ast_\tau w)_{i}$ as
$(\kappa\ast_\tau w)_{0}=0$ and
\begin{align*}
    (\kappa\ast_\tau w)_{i}:=\sum_{j=0}^{i }\tau\kappa_j w_{i -j},
  \quad \text{for} \ i=1,..,n,
\end{align*}
for any vector $(w_i)_{i=0}^n$, see \cite{Stefanelli2001}. Note in
particular that the minimization problem \eqref{eq: dmin} at level
$i=1$ is actually independent of $y_0$ as $(K_\tau \nabla y)_0=0$ by
definition. 

Under the convexity and continuity assumptions on the stored and
dissipation energy from the previous section, together with the
compactness from Lemma \ref{lemma:comp}, the existence of minimizers
for the problem \eqref{eq: dmin} is a straightforward implication of
the Direct Method of the calculus of variations. In particular, we
have the following. 
\begin{Lemma}[Existence of discrete solutions]\label{lemma:discrete_existence}
 There exist $(y_i,P_i)_{i=1}^n\in\mathcal{Q}^n$
 solving the minimization problems \eqref{eq: dmin}.
\end{Lemma}
\begin{proof} Fix  $i=1,\dots,N$ and assume $(y_j,P_j)_{j=0}^{i-1}\in
  \mathcal{Q}^{i}$ to be given. We aim at showing that the minimization problem
  \begin{equation}
    \min \left\{\int_\Omega W(\nabla y P^{-1},P,\nabla P)\,\d x-\langle
   l_i,y \rangle+{\mathcal D}((K_\tau\nabla
     y)_{i-1},P_{i-1},P)\right\}\label{eq:minproof}
\end{equation}
 admits a solution $(y,P)\in \mathcal{Q}$. Note that $(K_\tau\nabla
     y)_{i-1}$ depends just on the values $(y_j)_{j=0}^{i-1}\in
  \mathcal{Y}^{i}$ and it is hence known at level $i$.

Let $(y_k,P_k)\subset\mathcal{Q}$ be a minimizing
 sequence for \eqref{eq:minproof}. Without loss of generality we can
 assume that $\sup_k \mathcal{E}(t_i,y_k,P_k)<\infty$. 
Owing to the coercivity from Lemma \ref{lemma:comp}  one has that 
\begin{align*}
   \| \nabla y_k P^{-1}_n\|_{q_e}^{q_e}+  \|\nabla
  y_k\|_{q}^q+\|P_k\|_{{q_p} }^{q_p}+\|\nabla
  P_k\|_{{ q_ r }}^{ q_ r }\leq c_4( 1+\mathcal{E}(t_i,y_k,P_k))<c.
\end{align*}
Up to not relabeled subsequences, we hence have that $
y_k\rightharpoonup  y$ in $W^{1,q}(\Omega;\mathbb{R}^{3\times 3})$,
$\nabla y_kP^{-1}_k \rightharpoonup  F$ in $L^{q_e}(\Omega;\mathbb{R}^{3\times 3})$,
and $P_k\rightharpoonup P$ in $W^{1, q_ r }(\Omega;\SL)$. In particular,
possibly extracting again one finds $  P_k\to P $ , 
\UUU $P^{-1}_k \to P^{-1}$ in $C(\overline\Omega;\mathbb{R}^{3\times
  3})$. \EEE This implies that $F= \nabla y 
P^{-1}$, so that $\nabla y_k
P^{-1}_k\rightharpoonup \nabla y 
P^{-1}$ in $ L^{q_e}(\Omega;\mathbb{R}^{3\times 3})$. The
  claim follows from the lower semicontinuity of the
functional in \eqref{eq:minproof}. More precisely, the above proved
convergences, the polyconvexity
of $W_e$, and the lower semicontinuity of $W_p$ entail that
\begin{align*}
  \int_\Omega W(\nabla y P^{-1},P,\nabla P)\,\d x-\langle
   l_i,y \rangle\leq \liminf_{k\to \infty}\int_\Omega W(\nabla y_k P^{-1}_k,P_k,\nabla P_k)\,\d x-\langle
   l_i,y_k \rangle.
\end{align*}
Recalling the  continuity of $D$ one concludes that $(y,P)$
solves \eqref{eq:minproof} at level $i$.
\end{proof}

\subsection{A priori estimates} \label{sec:apriori} Let  
$(y_i,P_i)_{i=0}^{ n }\in {\mathcal Q}^{n+1}$ be a discrete solution and fix
$i=1,\dots,n$. Owing to the minimality from \eqref{eq: dmin} and the
triangle inequality \eqref{eq:D2}, for all
$(\hat{y},\hat{P})\in\mathcal{Q}$ one has that
\begin{align*}
  &\mathcal{E}(t_i,y_i,P_i) + {\mathcal D}((K_\tau\nabla
  y)_{i-1},P_{i-1},P_i) \leq  \mathcal{E}(t_i,\haz y ,\haz P) + {\mathcal D}((K_\tau\nabla
  y)_{i-1},P_{i-1},\haz P)\\
  &\quad \leq \mathcal{E}(t_i,\haz y ,\haz P) + {\mathcal D}((K_\tau\nabla
  y)_{i-1},P_{i-1}, P_i ) + {\mathcal D}((K_\tau\nabla
  y)_{i-1},P_{i},\haz P).
\end{align*}
In particular, one obtains the stability 
\begin{align}
  \mathcal{E}(t_i,y_i,P_i) \leq \mathcal{E}(t_i,\haz y ,\haz P) + {\mathcal D}((K_\tau\nabla
y)_{i-1},P_{i}, \haz P ).\label{eq:discr_stab}
\end{align}

Again minimality ensures the validity of the discrete upper energy estimate
\begin{align}
  &\mathcal{E}(t_i,y_i,P_i) + {\mathcal D}((K_\tau\nabla
  y)_{i-1},P_{i-1},P_i) \leq  \mathcal{E}(t_i,y_{i-1},P_{i-1}) \nonumber \\
  &\quad = \mathcal{E}(t_{i-1},y_{i-1},P_{i-1}) +
    \int_{t_{i-1}}^{t_i} \partial_s {\mathcal E}(s, y_{i-1},P_{i-1})\,\d s.\label{eq:discrete_energy}
\end{align}
By summing for $i=1, \dots,m\leq n$ one obtains that  
\begin{align}
  &\mathcal{E}(t_m,y_m,P_m) + \sum_{i=1}^m\mathcal{D} ((K_\tau\nabla
  y)_{i-1},P_{i-1},P_i)   \nonumber \\
  &\quad \leq   \mathcal{E}(0,y_{0},P_{0}) +  
    \sum_{i=1}^m\int_{t_{i-1}}^{t_i} \partial_s {\mathcal E}(s, y_{i-1},P_{i-1})\,\d s.\label{eq:discrete_energy}
\end{align}
The bound \eqref{eq:rateenergybnd} and an application of the discrete Gronwall
Lemma entail that
\begin{equation}
  \label{eq:bound}
\max_{m}\mathcal{E}(t_m,y_m,P_m) + \sum_{i=1}^n\mathcal{D} ((K_\tau\nabla
  y)_{i-1},P_{i-1},P_i)   \leq  c.
\end{equation}

\subsection{Convergence}\label{sec:convergence}
Let now the partitions $\sigma_n:=\{0=t_0^n<t_1^n<...<t_{n}^n=T\}$ be
given with $t_i^n = i \tau_n$, $\tau_n = T/n$, and
$i=1,\dots,n$. Owing to Lemma \ref{lemma:discrete_existence}, for all
$n$ one can find a discrete solution $(y_i,P_i)_{i=1}^n\in
\mathcal{Q}^{n}$. These satisfy  the discrete stability 
\eqref{eq:discr_stab},  as well as bound \eqref{eq:bound}.  In particular, the
corresponding  piecewise constant interpolants $(\overline
y_{\tau_n},\overline P_{\tau_n})$ satisfy for all $t\in [0,T]$ the stability inequality
\begin{align}
     &\mathcal{E}(\overline{t}_{ n}(t),\overline{y}_{ n}(t),\overline{P}_{n}(t))\leq
  \mathcal{E}(\overline{t}_{ n}(t),\hat{y},\hat{P})+
  \mathcal{D}(\overline{(K_{\tau_n}\nabla y)}_{n}(t-\tau_n),
  \overline{P}_{n}(t), \hat{P}) \nonumber\\
  &\qquad \forall (\hat y,\hat
  P)\in \mathcal{Q},\label{eq: istab}
\end{align}
which is nothing  but  
\eqref{eq:discr_stab}.
Moreover, the discrete upper energy estimate
\eqref{eq:discrete_energy} reads
\begin{align}\label{eq: iuee}
   & \mathcal{E}(\overline{t}_{ n}(t),\overline{y}_{n}(t),\overline{P}_{n}(t))+{\rm
  Diss}_{[0,t]}(\overline{(K_{\tau_n}\nabla
  y)}_{n}(\cdot -\tau_n),\overline{P}_{n})\\
  &\quad \leq
    \mathcal{E}(0,y_0,P_0) +
    \int_{0}^{\overline{t}_{ n}(t)}\partial_r\mathcal{E}(r,\overline{y}_{n}(r-\tau_n),\overline{P}_{n}(r-\tau_n))\,\d r,
\end{align}
for all $t\in\sigma_n$. Eventually, the bound
\eqref{eq:bound} can be rewritten as  
\begin{align}\label{eq:energybounds}
    \sup _{t\in [0,T]} \mathcal{E}(\overline{t}_n(t),\overline{y}_{n}(t),\overline{P}_{n}(t))+ {\rm
  Diss}_{[0,T]}(\overline{(K_{\tau_n}\nabla y)}_{n}(\cdot -\tau_n), \overline{P}_{n} )\leq c.
\end{align}
The coercivity from Lemma \ref{lemma:comp} entails that  
\begin{align}\label{eq:seqbnds}
     \sup_{t\in [0,T]}\left(\| \nabla \ove y_n(t)\ove P^{-1}_n(t)\|_{q_e}^{q_e}+ \|\nabla
  \ove y_n(t)\|_q^q +  \|\ove P_n(t)\|_{q_p}^{q_p}+
  \|\nabla \ove P_n(t)\|_{ q_ r }^{  q_ r  } \right)<c.
\end{align}
By possibly extracting not relabeled subsequences, the extended version of the generalised Helly principle from Theorem
\ref{thm:Helly} ensures that there exists a nondecreasing
function
\begin{equation}\label{eq:delta}
  \delta:[0,T]\to [0,\infty)
\end{equation}
such that 
\begin{align}
    &\big(\nabla \ove y_n \ove P^{-1}_n, \ove P_n, \nabla \ove
      P_n\big)\overset{Y}{\longrightarrow}
      \nu:=\left(\nu_{x,t},\lambda,\nu_{x,t}^\infty\right), \label{6.14} \\
    & \ove P_n(t)\to P(t) \quad \text{in} \
      L^\infty(\Omega;\mathbb{R}^{3\times 3})\quad \forall t \in[0,T], \label{eq:HellyA}\\
     &  {\rm Diss}_{[0,t]}(\overline{(K_{\tau_n}\nabla
       y)}_{n}(\cdot -\tau_n),\ove P_{n} )\to \delta(t)\quad \forall t \in [0,T], \label{eq:HellyB}\\[1mm]
    &   {\rm Diss}_{[s,t]}(  K\nabla y , P )\leq
      \delta(t)-\delta(s) \quad \forall 0\leq s<t\leq T, \label{eq:HellyC}
\end{align}
where the generated {Young measure} $\left(\nu_{x,t},\lambda,\nu_{x,t}^\infty\right)$ satisfies
\begin{align*}
    \sup_{t\in [0,T]}\int_\Omega\langle\nu_{x,t},| z_{e}
  |^{q_e}+| z_{p} |^{q_p}+| z_{ q_ r } |^{ q_ r }\rangle\, \d x < \infty.
\end{align*}
Note that the marginal of $\nu_{x,t}$ generated by 
	$(\nabla \ove P_n)$ is a $ q_ r $-{\it Young measure}, while the marginals generated 
	by $(\ove P_n)$ and $(\nabla \ove y_n \ove P^{-1}_n)$ are a
        $q_p$- and a $q_e$-{\it Young measure}, respectively. In
        particular, one has that
	\begin{align}
		\nabla \ove y_n \ove
          P^{-1}_n&\overset{\ast}{\rightharpoonup}\la\nu_{x,t},
                    z_e \ra=: \nabla y P^{-1}   \quad\text{in\:} L^{\infty}(0,T;L^{q_e}(\Omega)),\label{eq:01}\\
		\ove
          P_n&\overset{\ast}{\rightharpoonup}\la\nu_{x,t}, z_p \ra =: P  \quad\text{ in\:} L^{\infty}(0,T;L^{q_p}(\Omega)),\\
		\nabla \ove
          P_n&\overset{\ast}{\rightharpoonup}\la\nu_{x,t}, z_{q_r}
                \ra =: \nabla P \quad\text{in\:} L^{\infty}(0,T;L^{ q_ r }(\Omega)).\label{eq:03}
	\end{align}

 \subsection{Upper energy estimate}       \label{sec:energy-balance}
In order to prove that $(P,y,\nu)$ is indeed an energetic
measure-valued solution according to Definition \ref{def:ensol} we are
left with checking the energy
balance \eqref{def:energy_balance} and the  stability
\eqref{def:stability}. Let us start by proving an upper energy estimate.

For any $0\leq \hat t <T$ the discrete upper energy
  estimate \eqref{eq: iuee} can be rewritten as 
\begin{align}
    &\int_\Omega W(\nabla \ove y_n(\hat t) \ove P^{-1}_n(\hat t),\ove
  P_n(\hat t),\nabla \ove P_n(\hat t))\,\d x-\langle \ove l_n(\hat t),\ove y_n(\hat t) \rangle
      \nonumber\\[1mm]
  &\qquad + {\rm Diss}_{[0, \ove t_n(\hat t)]}(\overline{(K_{\tau_n}\nabla
  y)}_{n}(\cdot -\tau_n),\ove P_{n} )\nonumber\\[2mm]
    &\quad \leq \int_\Omega W(\nabla y_0  P^{-1}_0, 
      P_0,\nabla P_0)\,\d x-\langle \ove l(0),y_0\rangle -\int_{0}^{\ove t_n(\hat t)}\la\Dot{l}(r),\ove y_n(r-\tau_n)\ra \,\d r.\label{eq:dopo}
\end{align}

From the very definition of constant interpolants we have
\begin{align*}
    {\rm Diss}_{[0, \ove t_n(\hat t)]}(\overline{(K_{\tau_n}\nabla
  y)}_{n}(\cdot -\tau_n),\ove P_{n} )={\rm Diss}_{[0,\hat t]}(\overline{(K_{\tau_n}\nabla
  y)}_{n}(\cdot -\tau_n),\ove P_{n}).
\end{align*}

The forcing terms can be rewritten as 
\begin{align*}
    \int_{0}^{\ove t_n(\hat t) }\la\Dot{l}(r),\ove y_n(r-\tau_n)\ra \,\d r=\int_{0}^{\hat
  t}\la\Dot{l}(r),\overline y_n(r-\tau_n)\ra \,\d r+I_n(\hat t)
\end{align*}
where the residual term $I_n(\hat t)$ is given by
\begin{align*}
 I_n(\hat t)=\int_{\hat t}^{\ove t_n(\hat t)}\la\Dot{l}(r),\ove y_n(r-\tau_n)\ra \,\d r.
\end{align*}
Let us remark that for all $  \hat t\in [0,T]$ one has
$ I_n(\hat t)\to 0$ as $n\to \infty$. Indeed, owing to the
regularity of $\dot l$ and the energy bounds \eqref{eq:seqbnds} we infer
\begin{align}
  | I_n(\hat t)| &=
    \left| \int_{\hat t}^{\ove t_n(\hat t)}\la\Dot{l}(r),y_n(r-\tau_n)\ra\,
    \d r\right| \leq c
       \int_{\hat
    t}^{\ove t_n(\hat t)}\|\Dot{l}(r)\|_{\mathcal{Y}^\ast} \,\d
                   r \leq c\tau_n\to 0 \label{eq:zeroterms}
\end{align}
as $n\to\infty$.

The discrete upper energy estimate \eqref{eq:dopo} can hence be
rewritten as 
\begin{align}
    &\int_\Omega W(\nabla \ove y_n(\hat t) \ove P^{-1}_n(\hat t),\ove
  P_n(\hat t),\nabla \ove P_n(\hat t))\,\d x-\langle \ove l_n(\hat t),\ove y_n(\hat t) \rangle
      \nonumber\\[1mm]
  &\qquad + {\rm Diss}_{[0,\hat t]}(\overline{(K_{\tau_n}\nabla y)}_{n}(\cdot -\tau_n),\ove
  P_{n} )\nonumber\\[2mm]
   &\quad  \leq  \int_\Omega W(\nabla y_0  P^{-1}_0, 
      P_0,\nabla P_0)\,\d x-\langle \ove l(0),y_0\rangle
     -\int_{0}^{\hat t}\la\Dot{l}(r),\ove y_n(r-\tau_n)\ra \,  \d r
     + I_n(\hat
     t) . \label{eq:dopo2}
\end{align}

Fix now $ t \in (0,T)$. Along the argument, the choice
of $t$ will be progressively restricted. We anticipate that,
at each step, a null set of points $t$ is
discarded. 

Integrate  
\eqref{eq:dopo2} with respect to $\hat t$ over $ \in (t,t+\epsi)$ for some $\epsi \in
(0,T-t)$. By further dividing by $\epsi$ we obtain
\begin{align}
    &\frac{1}{\epsi}\int_{t}^{t+\epsi}\!\! \int_\Omega W(\nabla \ove
      y_n(\hat t) \ove P^{-1}_n(\hat t),\ove
  P_n(\hat t),\nabla \ove P_n(\hat t))\,\d x\, \d \hat
      t-\frac{1}{\epsi}\int_{t}^{t+\epsi}\langle \ove l_n(\hat t),\ove
      y_n(\hat t) \rangle \,\d\hat t\nonumber\\
  &\qquad 
  + \frac{1}{\epsi}\int_{t}^{t+\epsi} {\rm Diss}_{[0,\hat
    t]}(\overline{(K_{\tau_n}\nabla y)}_{n}(\cdot -\tau_n),\ove
  P_{n} )\,\d \hat t\nonumber\\
   &\quad  \leq  \int_\Omega W(\nabla y_0  P^{-1}_0, 
      P_0,\nabla P_0)\,\d x-\langle \ove l(0),y_0\rangle
     \nonumber\\
  &\qquad -\frac{1}{\epsi}\int_{t}^{t+\epsi}\!\! \int_{0}^{\hat t}\la\Dot{l}(r),\ove y_n(r-\tau_n)\ra
    \,\d r\, \d\hat t+ \frac{1}{\epsi}\int_{t}^{t+\epsi}  I_n(\hat t)\,\d \hat t. \label{eq:dopo33}
\end{align}

We next pass to the limit term by term in \eqref{eq:dopo33}, first for $n\to
\infty$, and subsequently for $\epsi \to 0$. The properties of the
Young measure $\nu$ and the  structural result from Lemma
\ref{lemma: spartYM} guarantee that  the concentration measure $\lambda$ admits a disintegration of the
form $\lambda(\d x,\d t)=\eta_t(\d x)\otimes \d t$, where $t\mapsto \eta_t$ is a bounded measurable map from $[0,T]$ into
 $\mathcal{M}^+(\overline{\Omega})$. In particular, we have 
\begin{align} 
    &\lim_{\epsi \to 0} \lim_{n\to \infty}\frac{1}{\epsi}\int_{t}^{t+\epsi}\!\! \int_\Omega W(\nabla \ove y_n(\hat t) \ove P^{-1}_n(t),\ove
      P_n(\hat t),\nabla \ove P_n(\hat t))\,\d x\, \d \hat t\nonumber \\
  &\quad =\int_\Omega \la\nu_{x,t},\tilde
    W( z_e, z_{ q_ r } )\ra +W_p(P(t))\,\d x +
    \int_{\overline{\Omega}} \la\nu_{x,t}^\infty,\tilde
    W^\infty( z_e, z_{ q_ r } )\ra \,\d\eta_{t}(x) \label{eq:dopo3}
\end{align}
for a.e. $t\in (0,T)$.
The specific form of the plastic term ensues from the pointwise $L^\infty$ strong
convergence of $\ove P_n(t)$ and an application of the Dominated
Convergence Theorem. In particular, one has that
$$\lim_{\epsi \to 0} \lim_{n\to \infty}\frac{1}{\epsi}\int_{t}^{t+\epsi}\!\! \int_\Omega W_p(\ove P_n(\hat
t) )\,\d x \, \d \hat t = \int_\Omega W_p( P(
t) )\,\d x.$$

By further restricting to Lebesgue points of the   function $t\mapsto
\la l(t), y(t)\ra$, the regularity of $l$ entails that, for a.e. $ 
t \in (0,T)$,
\begin{align}
  &\lim_{\epsi \to 0} \lim_{n\to \infty}\frac{1}{\epsi}\int_{t}^{t+\epsi} \la \ove l_n (\hat t), \ove
    y_n(\hat t)\ra \,\d\hat t = \la l(t), y(t)\ra.\label{eq:L1}
\end{align}

For fixed $\epsi$, the limit $n\to \infty$ in the dissipation term reads
\begin{align*}
  &\lim_{n\to \infty}\frac{1}{\epsi}\int_{t}^{t+\epsi} {\rm Diss}_{[0,\hat t]}(\overline{(K_{\tau_n}\nabla y)}_{n}(\cdot -\tau_n),\ove
  P_{n} )\,\d \hat t \stackrel{\eqref{eq:HellyB}}{=} \frac{1}{\epsi}\int_{t}^{t+\epsi} \delta(\hat
    t)\,\d \hat t - \delta(0).
\end{align*}
Hence, by taking $\epsi\to 0$, for a.e. $ t \in (0,T)$ one gets
\begin{align}
  &\lim_{\epsi\to 0}\lim_{n\to \infty}\frac{1}{\epsi}\int_{t}^{t+\epsi} {\rm Diss}_{[0,\hat t]}(\overline{(K_{\tau_n}\nabla y)}_{n}(\cdot -\tau_n),\ove
    P_{n} )\,\d \hat t =   \delta(t) - \delta(0) \stackrel{\eqref{eq:HellyC}}{\geq} {\rm Diss}_{[0,t]}(K\nabla y,
    P ).\nonumber
\end{align}

Owing to to the continuity of $\Dot{l}$ one can compute that
\begin{align}
  &\lim_{\epsi\to 0}\lim_{n\to \infty}\frac{1}{\epsi}\int_{t}^{t+\epsi}\!\! \int_{0}^{\hat t}\la\Dot{l}(r),\ove y_n(r-\tau_n)\ra
    \,\d r\, \d\hat t \nonumber\\
  &\quad= \lim_{\epsi\to 0}\lim_{n\to
  \infty}\frac{1}{\epsi}\int_{t}^{t+\epsi}\!\! \int_{-\tau_n}^{\hat
  t-\tau_n}\la\Dot{l}(r+\tau_n),\ove y_n(r)\ra
  \,\d r\, \d\hat t \nonumber\\
  & \quad= \int_0^t \la \Dot{l}(r), y(r)\ra
    \,\d r. \label{eq:II}
\end{align}
By using estimate \eqref{eq:zeroterms} we readily get that
\begin{equation}\lim_{\epsi\to 0}\lim_{n\to
  \infty} \frac{1}{\epsi}\int_{t}^{t+\epsi}   I_n(\hat t)\,\d \hat t   =0.\label{eq:I}
\end{equation}
 
The upper energy estimate
\begin{align}
  &\la\!\la\nu_t,W\ra\!\ra-\langle l(t),y(t)\rangle+{\rm
                                       Diss}_{[0,t]}(K\nabla y,P) \nonumber\\
  &\quad \leq \int_\Omega W(\nabla y_0 P^{-1}_0, P_0, \nabla P_0)\,\d x-\langle l(0),y_0 \rangle -\int_0^t
    \la \dot l(r),y(r)\ra \,\d r \label{eq:oppo}
    \end{align}
for a.e. $t\in (0,T) $ follows by collecting
\eqref{eq:dopo3}--\eqref{eq:I} and recalling definition
\eqref{eq:action}.

\subsection{Stability}\label{sec:stability}
To check the stability \eqref{def:stability} we fix
$(\hat y, \hat P)\in \mathcal{Q}$ and integrate the discrete stability
\eqref{eq: istab} over $(t,t+\epsi)$ for $t\in (0,T)$ given and
$\epsi \in (0,T-t)$. By further dividing by $\epsi$ we get
\begin{align}
  &\frac{1}{\epsi}\int_{t}^{t+\epsi}\!\! \int_\Omega W(\nabla \ove
      y_n(\hat t) \ove P^{-1}_n(\hat t),\ove
  P_n(\hat t),\nabla \ove P_n(\hat t))\,\d x\, \d \hat
      t-\frac{1}{\epsi}\int_{t}^{t+\epsi}\langle \ove l_n(\hat t),\ove
      y_n(\hat t) \rangle \,\d\hat t\nonumber\\
   &\quad  \leq \frac{1}{\epsi}\int_{t}^{t+\epsi}\!\! \int_\Omega
     W(\nabla \hat y  \hat  P^{-1} ,\hat P,\nabla \hat P)\,\d x\, \d \hat
     t- \frac{1}{\epsi}\int_{t}^{t+\epsi}\langle \ove l_n(\hat  t ),\hat y
     \rangle \,\d\hat t \nonumber \\
  &\qquad+  \frac{1}{\epsi}\int_{t}^{t+\epsi}\!\!
     \!\!\int_\Omega D(\overline{(K_{\tau_n} \nabla y)}_n(\hat t-\tau_n),\ove
     P_n(\hat t ),\hat P) \,\d \hat t. \label{eq:dopo5}
\end{align}
By taking $n\to \infty$ first and then $\epsi \to 0$, the convergences
\eqref{eq:dopo3} and \eqref{eq:L1} ensure that the  left-hand  side of
\eqref{eq:dopo5} converges to the  left-hand  side of
\eqref{def:stability} for a.e. $t\in(0,T)$. On the other hand, one readily checks that 
\begin{align} 
 &\lim_{\epsi\to 0}\lim_{n\to \infty} \left( \frac{1}{\epsi}\int_{t}^{t+\epsi}\!\! \int_\Omega
     W(\nabla \hat y  \hat  P^{-1} ,\hat P,\nabla \hat P)\,\d x\, \d \hat
     t- \frac{1}{\epsi}\int_{t}^{t+\epsi}\langle \ove l_n(\hat t),\hat y
  \rangle \,\d\hat t \right)\nonumber\\
  &\quad =\int_\Omega
     W(\nabla \hat y  \hat  P^{-1} ,\hat P,\nabla \hat P)\,\d x 
     -  \langle \l(t),\hat y
  \rangle  \label{eq:dopo6}
\end{align}
for a.e. $t\in (0,T)$.

As regards the dissipation term, let us start by computing
\begin{align}
&\left| D(\overline{(K_{\tau_n} \nabla y)}_n(\hat t-\tau_n),\ove
     P_n(\hat t ),\hat P) - D( K \nabla y)(\hat t), P(\hat t ),\hat
  P)\right|\nonumber\\
  & \ \quad \leq \left| D(\overline{(K_{\tau_n} \nabla y)}_n(\hat t-\tau_n),\ove
    P_n(\hat t ),\hat P) - D(K\nabla y(\hat t),\ove
    P_n(\hat t ),\hat P) \right|\nonumber\\
    &\qquad + \left| D(K\nabla y(\hat t),\ove
    P_n(\hat t ),\hat P) -  D( K \nabla y(\hat t), P(\hat t ),\hat
      P)\right|\nonumber\\
  &\quad \stackrel{\eqref{eq:D4}}{\leq} c |\overline{ (K_{\tau_n} \nabla
    y)}_n(\hat t-\tau_n)- K\nabla y(\hat t)|\,\haz D(\ove P_n(\hat t),
    \hat P) \nonumber\\
    &\qquad + \left| D(K\nabla y(\hat t),\ove
    P_n(\hat t ),\hat P) -  D( K \nabla y (\hat t), P(\hat t ),\hat
      P)\right|\nonumber\\
  & \stackrel{\eqref{eq:D11}, \eqref{eq:D3}}{\leq} c | \overline{ (K_{\tau_n} \nabla
    y)}_n (\hat t-\tau_n)- K\nabla y(\hat t)|\, (1 + |\ove P_n(\hat
    t)| + | \hat P| )     \nonumber\\
    &\qquad + \left| D(K\nabla y(\hat t),\ove
    P_n(\hat t ),\hat P) -  D( K \nabla y (\hat t), P(\hat t ),\hat
      P)\right|.
  \label{eq:dopo7}
\end{align}
Using \cite[Prop. 4.4]{Stefanelli2001} we can control
\begin{align}
  &\|\overline{ (K_{\tau_n} \nabla
    y)}_n(\cdot -\tau_n)- K\nabla y\|_{L^1(0,T;L^q(\Omega;\Rz^{3\times
  3}))}\nonumber\\
  &\quad \leq \tau_n c (1+\|\nabla \ove y_n\|_{L^1(0,T;L^q(\Omega;\Rz^{3\times
  3}))}) + \| K\nabla (\ove y_n - y)\|_{L^1(0,T;L^q(\Omega;\Rz^{3\times
  3}))}.\label{eq:vol}
\end{align}
Recalling that $\ove P_n^{-1}$ is uniformly bounded in space and time,
by integrating \eqref{eq:dopo7} on $\Omega$ and $(t,t+\epsi)$ we get
\begin{align}
  &\frac{1}{\epsi}\int_{t}^{t+\epsi}\!\! \int_\Omega \left| D(\overline{(K_{\tau_n} \nabla y)}_n(\hat t-\tau_n),\ove
     P_n(\hat t ),\hat P) - D( K \nabla y)(\hat t), P(\hat t ),\hat
  P)\right|\d x\,\d \hat t \nonumber\\
  &\quad \stackrel{\eqref{eq:dopo7}}{\leq} \tau_n c  (1+\|\nabla \ove y_n\|_{L^1(0,T;L^q(\Omega;\Rz^{3\times
  3}))}) + c  \| K\nabla (\ove y_n - y)\|_{L^1(0,T;L^q(\Omega;\Rz^{3\times
    3}))} \nonumber\\
  &\qquad+  \frac{1}{\epsi}\int_{t}^{t+\epsi}\!\! \int_\Omega \left|D(K\nabla y(\hat t),\ove
    P_n(\hat t ),\hat P) -  D( K \nabla y(\hat t), P(\hat t ),\hat
      P)\right|\d x\,\d \hat t.\label{eq:dopo8}
\end{align}
As $K\nabla (\ove y_n - y)\to 0$ in $L^1(0,T;L^q(\Omega;\Rz^{3\times
  3}))$ as effect of the compact operator $K$, by first taking the limit
$n\to \infty$, also using the continuity of $D$ as well as the
bounds \eqref{eq:D3}, and then the limit $\epsi \to 0$ in \eqref{eq:dopo8} we obtain  
\begin{align}
   &\lim_{\epsi \to 0}\lim_{n\to \infty} \frac{1}{\epsi}\int_{t}^{t+\epsi}\!\! \int_\Omega D(\overline{(K_{\tau_n} \nabla y)}_n(\hat t-\tau_n),\ove
  P_n(\hat t ),\hat P)\,\d x\, \d \hat t \nonumber\\
  &\quad = \lim_{\epsi \to 0}\lim_{n\to \infty} \frac{1}{\epsi}\int_{t}^{t+\epsi}\!\! \int_\Omega D(K\nabla y(\hat t),\ove
    P_n(\hat t ),\hat P) \,\d x\, \d \hat t = \int_\Omega D( K \nabla y (t), P( t ),\hat
      P)\,\d x.\nonumber
\end{align}
Owing to \eqref{eq:dopo6} and \eqref{eq:dopo8}, inequality
\eqref{eq:dopo5} corresponds to the stability \eqref{def:stability}.

\subsection{Lower energy estimate}\label{sec:lower}
We are left with proving the converse inequality to \eqref{eq:oppo},
namely,
\begin{align}
  &\la\!\la\nu_t,W\ra\!\ra-\langle l(t),y(t)\rangle+{\rm
                                       Diss}_{[0,t]}(K\nabla y,P) \nonumber\\
  &\quad \geq \int_\Omega W(\nabla y_0 P^{-1}_0, P_0, \nabla P_0)\,\d x-\langle l(0),y_0 \rangle -\int_0^t
    \la \dot l(r),y(r)\ra \,\d r \label{eq:oppo2}
    \end{align}
    for a.e. $t\in (0,T) $. This classically follows from stability
\eqref{def:stability}, see \cite[Prop.~5.7]{Mielke2005a}. As the map
$t \in [0,T] \mapsto \la \Dot{l}(t),y(t)\ra $
is integrable, one can use Lemma \ref{lemma:DalMaso}, which is a
straightforward extension of 
\cite[Lemma 4.12]{DalMaso2004} and, given any null set $N\subset (0,T)$ find sequence of partitions $\{0=s_0^m<\dots
<s_{M^m}^m=t\}$, $t \in (0,T)$, indexed by $m\in \Nz$ with $\eta_m:=\max_{j=1,\dots,M^m}(s_j^m
- s_{j-1}^m) \to 0$ as $m\to \infty$, such that $(s_j^m)_{j=0}^{M^m}
\not \in N$, and the Lebesgue integral can be approximated by
the corresponding Riemann sum, namely,
\begin{align}
  \lim_{m\to \infty }\sum_{j=1}^{M^m} \la \Dot{l}(s^m_j), y(s^m_j)\ra
  (s^m_{j}-s^m_{j-1}) = \int_0^t \la \Dot{l}(r),y(r)\ra \,\d r.
  \label{eq:per_en4}
\end{align}
In particular, $N$ can be taken to be the complement of the 
points  in
$(0,T)$ where the stability \eqref{def:stability} and
convergences properties used below hold.

Fix $j=1,\dots
,M^m$, $\epsi\in (0,T-t)$, and choose $(\hat y,\hat P) =
(y(\hat t),P(\hat t))$ with $\hat t\in (s_{j }^m,s_{j}^m+\epsi)$ in
\eqref{def:stability} at $s_{j-1}^m$, and
integrate on $(s_{j}^m,s_{j}^m+\epsi)$ getting
\begin{align*}
  &\la\!\la\nu_{ s_{j-1}^m},W\ra\!\ra-\langle l(s_{j-1}^m),y (s_{j-1}^m)
                       \rangle  \leq \frac{1}{\epsi}\int_{s_{j}^m}^{s_{j}^m+\epsi}\!\!
  \int_\Omega W(\nabla y(\hat t) {P}^{-1}(\hat t), P(\hat t),\nabla
  P(\hat t))\,\d x\, \d \hat t\nonumber\\&\quad - \frac{1}{\epsi}\int_{s_{j}^m}^{s_{j}^m+\epsi}\langle l(s_{j-1}^m), 
  y(\hat t)\rangle\, \d \hat t
  + \frac{1}{\epsi}\int_{s_{j}^m}^{s_{j}^m+\epsi}  {\mathcal D} ( K\nabla y(s_{j-1}^m),
    P(s_{j-1}^m), P(\hat t))\, \d \hat t. 
\end{align*}
We now assume that the points
$(s_j^m)_{j=1}^{M^m}$ are such that the above right-hand side converges as
$\epsi \to 0$,
along the lines of Subsection \ref{sec:stability}. By passing to the
limit we then get
\begin{align}
  &\la\!\la\nu_{ s_{j-1}^m},W\ra\!\ra-\langle l(s_{j-1}^m),y (s_{j-1}^m)
                       \rangle  \nonumber \\
  &\quad \leq  
    \la\!\la\nu_{ s_{j}^m},W\ra\!\ra  -  \langle l(s_{j-1}^m), 
  y(s_{j}^m)\rangle 
  +   {\mathcal D} ( K\nabla y(s_{j-1}^m),
    P(s_{j-1}^m), P(s_{j}^m))  \nonumber\\
  &\quad = \la\!\la\nu_{ s_{j}^m},W\ra\!\ra  -  \langle l(s_{j}^m), 
  y(s_{j}^m)\rangle + {\mathcal D} ( K\nabla y(s_{j-1}^m),
    P(s_{j-1}^m), P(s_{j}^m))
   \nonumber  \\
  &\qquad + \langle  l(s_{j}^m)-  l(s_{j-1}^m),  y(s_{j}^m)\rangle\label{eq:sum1}
\end{align}
An analogous estimate can be obtained for $j=0$, by using the
stability of the initial data\eqref{eq:stab0}. In particular, we have
that
\begin{align}
  &\int_\Omega W(\nabla y_0 P^{-1}_0, P_0, \nabla P_0)\,\d x-\langle
    l(0),y_0 \rangle \nonumber \\
  &\quad \leq    \la\!\la\nu_{ s_{1}^m},W\ra\!\ra  -  \langle l(s_{1}^m), 
  y(s_{1}^m)\rangle + {\mathcal D} ( K\nabla y(0),
    P_0, P(s_{1}^m))
   \nonumber  \\
  &\qquad + \langle  l(s_{1}^m)-  l(0),  y(s_1^m)\rangle\label{eq:sum2}
\end{align}
Summing up \eqref{eq:sum1} for $j= 2,  \dots,M^m$ and 
adding it to  \eqref{eq:sum2} one obtains
\begin{align}
  &\int_\Omega W(\nabla y_0 P^{-1}_0, P_0, \nabla P_0)\,\d x-\langle
    l(0),y_0 \rangle  \leq  
    \la\!\la\nu_{ t},W\ra\!\ra  -  \langle l(t), 
    y(t )\rangle \nonumber \\
  &\qquad  + \sum_{j=1}^{M^m}  {\mathcal D} ( K\nabla y(s_{j-1}^m),
    P(s_{j-1}^m), P(s_{j}^m))+ \sum_{j=1}^{M^m} \langle  l(s_{j}^m)-  l(s_{j-1}^m),  y(s_{j}^m)\rangle
 . \label{eq:oppo3}
\end{align}

To control the dissipation term in the right-hand side of
\eqref{eq:oppo3} we compute
\begin{align} 
  &\sum_{j=1}^{M^m}D(K\nabla y(s^m_{j-1}),
    P(s^m_{j-1}),P(s^m_{j}))\nonumber\\
  &\quad \stackrel{\eqref{eq:D0}}{=}\sum_{j=1}^{M^m}D(K\nabla y(s^m_{j}),
    P(s^m_{j}),P(s^m_{j-1})) \nonumber\\
  &\qquad + \sum_{j=1}^{M^m}\left(D(K\nabla y(s^m_{j-1}),
    P(s^m_{j}),P(s^m_{j-1}))   - D(K\nabla y(s^m_{j}),
    P(s^m_{j}),P(s^m_{j-1})) \right)\nonumber\\
  &\quad \stackrel{\eqref{eq:D4}}{\leq}\sum_{j=1}^{M^m}D(K\nabla y(s^m_{j}),
    P(s^m_{j}),P(s^m_{j-1})) \nonumber\\
  &\qquad + c\sum_{j=1}^{M^m}|K\nabla y(s_j^m) - K\nabla
    y(s_{j-1}^m)|\, \haz D( P(s_j^m), P(s_{j-1}^m))\nonumber\\
  &\quad \stackrel{\eqref{eq:D11}}{\leq}\mathrm{Diss}_{[0,t]}(K\nabla
    y,P) + c\sum_{j=1}^{M^m}(s_j^m - s_{j-1}^m) D(K\nabla y(s_{j}^m),P(s_j^m), P(s_{j-1}^m))\nonumber\\
  &\quad \leq\mathrm{Diss}_{[0,t]}(K\nabla
    y,P) + c\eta_m \mathrm{Diss}_{[0,t]}(K\nabla
    y,P).\nonumber
\end{align}
In particular, by passing to the $\limsup$ for $m\to \infty $ we get that
\begin{equation}
  \label{eq:per_en60}
  \limsup_{m\to \infty}\sum_{j=1}^{M^m}D(K\nabla y(s^m_{j-1}),
    P(s^m_{j-1}),P(s^m_{j})) \leq \mathrm{Diss}_{[0,t]}(K\nabla
    y,P).
\end{equation}

In order to pass to the limit as $m\to \infty$ in the last term in the
right-hand side of inequality \eqref{eq:oppo3} we write
\begin{align}
  &\sum_{j=1}^{M^m}\la  l(s^m_{j})-l(s^m_{j-1}), y(s^m_{j})\ra- \int_0^t\la \Dot{l}(r),y(r)\ra \,\d r \nonumber\\
  &\quad=\left(\sum_{j=1}^{M^m}  \la l(s_j^m) - l(s_{j-1}^m) - \Dot{l}(s_j^m)(s_j^m - s_{j-1}^m),
    y(s_j^m)\ra\right) \nonumber\\
  &\qquad+ \left(\sum_{j=1}^{M^m}  \la \Dot{l}(s_j^m) (s_j^m - s_{j-1}^m),
    y(s_j^m)\ra  - \int_0^t \la \Dot{l}(r),y(r)\ra \,\d r\right) =:
    I^m_1+ I^m_2.\nonumber
\end{align}
Relation \eqref{eq:per_en4} corresponds to $\lim_{m\to \infty}
I_2^m=0$. On the other hand, one can find $r^m_j \in
(s_{j-1}^m,s_j^m)$ such that 
\begin{align*}
  |I_1^m| = \left|\sum_{j=1}^{M^m}  \la (\Dot{l}(r_j^m)- \Dot{l}(s_j^m))(s_j^m - s_{j-1}^m),
    y(s_j^m)\ra\right| \leq   \omega (\eta_m) t \| y \|_{L^\infty(0,T;\mathcal{Y})}
\end{align*}
where $\omega$ is a modulus of continuity for $\Dot{l}:[0,T]\to
\mathcal{Y}^*$. This implies that $\lim_{m\to \infty}
I_1^m=0$, as well. In particular, we have that
\begin{equation}
  \label{eq:per_en61}
 \lim_{m\to \infty }\sum_{j=1}^{M^m}\la  l(s^m_{j})-l(s^m_{j-1}),
 y(s^m_{j})\ra = \int_0^t\la \Dot{l}(r),y(r)\ra \,\d r.
\end{equation}
Using this and \eqref{eq:per_en60} we can hence pass to
the $\limsup$ as $m\to \infty$ in relation \eqref{eq:oppo3} and obtain
the lower energy estimate \eqref{eq:oppo2}.

This concludes the proof of Theorem \ref{thm:existencemv}.


\section{Proof of Proposition
  \ref{lemma:functions}:  energetic function-valued solutions}\label{proof:lemma:functions}
Following the arguments of Section \ref{sec:convergence}, we find not
relabeled subsequences such that $\ove y_n
\stackrel{\ast}{\rightharpoonup}y$ in
$L^\infty(0,T;\mathcal{Y})$ and $\ove P_n(t)\to P(t)$ in
$L^\infty(\Omega;\Rz^{3\times 3})$ for all $t\in [0,T]$. On the other
hand, given the $L^\infty(0,T;W^{1, q_ r }(\Omega;\Rz^{3\times 3\times
  3}))$ bound, one can assume, with no need of further extractions,
that $\nabla \ove P_n(t) \rightharpoonup \nabla P(t)$ in $L^{ q_ r }(\Omega;\Rz^{3\times 3\times
  3})$ for a.e. $t\in (0,T)$. For all $t\in [0,T]$ one can find
a subsequence $\ove y_{n_k^t}(t)$ (possibly depending on $t$) and
$\tilde y \in  {\mathcal Y}$ such that
$y_{n_k^t}(t)\rightharpoonup\tilde y=:\tilde y (t)$ in $
 {\mathcal Y}$. In particular, one has that
$$\nabla \ove y_{n_k^t} (t)\ove P^{-1}_{n_k^t}(t)\rightharpoonup
\nabla \tilde y(t){P}^{-1}(t)\ \ \text{in} \ \
L^{q_e}(\Omega;\Rz^{3\times 3}).$$
Hence, owing to polyconvexity, for a.e. $t\in(0,T)$ one has
\begin{align}&\int_\Omega W(\nabla \tilde y(t){P}^{-1}(t), P(t),\nabla
P(t))\,\d x-\langle l(t), \tilde y(t)\rangle\nonumber\\
&\quad\leq \liminf_{k\to  \infty}
\int_\Omega W(\nabla \ove y_{n_k^t} (t)\ove P^{-1}_{n_k^t}(t), \ove P_{n_k^t}(t),\nabla
\ove P_{n_k^t}(t))\,\d x-\langle \ove l_{n_k^t}(t), \ove y_{n_k^t}
                                             (t)\rangle.\label{eq:per_stab1}
\end{align}
On the other hand, arguing as in \eqref{eq:dopo7} we have that
\begin{align}
  &D(\overline{(K_{\tau_{n_k^t}}\nabla y)}_{n_k^t}(t-\tau_{n_k^t}),\ove
  P_{n_k^t}(t),\hat P ) \to D(K\nabla
    y(t),P(t),\hat P). \label{eq:per_stab2}
\end{align}
Owing to \eqref{eq:per_stab1}--\eqref{eq:per_stab2} we can pass to
$\liminf$ in the discrete stability \eqref{eq: istab} as $k\to \infty$
and obtain the stability \eqref{eq:stab1}.

Choose now $\hat s=0$, $\hat t =t$, and let $n=n^t_k$ in \eqref{eq:dopo} getting
\begin{align}
    &\int_\Omega W(\nabla \ove y_{n^t_k}( t) \ove P^{-1}_{n^t_k}( t),\ove
  P_{n^t_k}( t),\nabla \ove P_{n^t_k}( t))\,\d x-\langle \ove l_{n^t_k}( t),\ove y_{n^t_k}( t) \rangle
  \nonumber \\
  &\qquad + {\rm Diss}_{[0,\underline{ t}_{n^t_k}]}(\overline{(K_{\tau_{n^t_k}}\nabla
    y)}_{{n^t_k}}(\cdot - \tau _{n^t_k}),\ove P_{{n^t_k}} ) \nonumber\\
  &\quad \leq \int_\Omega W(\nabla y_0 P^{-1}_0 , P_0, \nabla P_0)\,\d x-\langle \ove l_0,
      y_0\rangle  -\int_{0}^{\underline{ t}_{n^t_k}}\la\Dot{l}(r),\ove y_{n^t_k}(r-\tau_{n^t_k})\ra \,\d r.\label{eq:per_en1}
\end{align}

Fix a partition $\{0=s_0<\dots<s_M=t\}$ of $[0,t]$.
The continuity of $D$ ensures that
\begin{align}
  &\sum_{j=1}^M D(K\nabla y(s_j),P(s_j),P(s_{j-1})) \nonumber\\
  &\leq \liminf_{k\to \infty}\sum_{j=1}^M  D(\overline{(K_{\tau_{n_k^t}}\nabla y)}_{n_k^t}(s_j-\tau_{n_k^t}),\ove
    P_{n_k^t}(s_j),\ove P_{n_k^t}(s_{j-1}) ) \nonumber\\
  &\leq \liminf_{k\to \infty} {\rm Diss}_{[0,\underline{ t}_{n^t_k}]}(\overline{(K_{\tau_{n^t_k}}\nabla
    y)}_{{n^t_k}}(\cdot - \tau _{n^t_k}),\ove P_{{n^t_k}} ).\nonumber
\end{align}
By passing to the supremum with respect to the partitions of $[0,t]$ we have hence
checked that
\begin{align}
   {\rm Diss}_{[0,t]}(K\nabla y,P)
  \leq \liminf_{k\to \infty} {\rm Diss}_{[0,\underline{ t}_{n^t_k}]}(\overline{(K_{\tau_{n^t_k}}\nabla
    y)}_{{n^t_k}}(\cdot - \tau _{n^t_k}),\ove P_{{n^t_k}} ). \label{eq:per_en2}
\end{align}
As for the forcing term, by passing to the limit $k\to \infty$ and
using the continuity of $\Dot{l}$ we have 
\begin{align}
  &\int_{0}^{\underline{ t}_{n^t_k}}\la\Dot{l}(r),\ove
  y_{n^t_k}(r-\tau_{n^t_k})\ra \,\d r =\int_{0}^{\tau_{n^t_k}}\la\Dot{l}(r),\ove
  y_{n^t_k}(r-\tau_{n^t_k})\ra \,\d r +\int_{\tau _{n^t_k}}^{\underline{ t}_{n^t_k}}\la\Dot{l}(r),\ove
  y_{n^t_k}(r-\tau_{n^t_k})\ra \,\d r\nonumber\\
  &\quad 
  = \int_0^{\tau_{n^t_k}}\la \Dot{l}(r),y_0\ra \,\d r + \int_{0}^{\underline{ t}_{n^t_k}
  - \tau_{n^t_k} }\la\Dot{l}(r+\tau_{n^t_k}),\ove
  y_{n^t_k}(r)\ra \,\d r \to \int_0^t \la \Dot{l}(r),y(r)\ra \,\d r.\label{eq:per_en3}
\end{align}
Owing to \eqref{eq:per_stab1} and
\eqref{eq:per_en2}--\eqref{eq:per_en3} we can pass to the $\liminf$ as
$k\to \infty$ in the discrete upper energy inequality
\eqref{eq:per_en1} and obtain
\begin{align}
 & \int_\Omega W(\nabla \tilde y(t) P^{-1}(t), P(t),\nabla
                                  P(t))\,\d x-\langle l(t),\tilde
   y(t)\rangle \nonumber  +{\rm
   Diss}_{[0,t]}(K\nabla y, P )\nonumber\\
  &\quad \leq \int_\Omega W(\nabla y_0P^{-1}_0, P_0,\nabla
    P_0)\,\d x -\langle l(0), y_0\rangle 
    -\int_0^t \la \dot l(r),y(r)\ra \,\d r\label{eq:upper_en}
\end{align}
for a.e. $t \in (0,T)$.

\UUU In case $y =\tilde y$ a.e., the \EEE converse inequality follows from the stability
\eqref{eq:stab1}. Indeed, using Lemma \ref{lemma:DalMaso} \UUU we can
find a \EEE sequence of partitions $\{0=s_0^m<\dots
<s_{M^m}^m=t\}$ indexed by $m\in \Nz$ with $\eta_m:=\max_{j=1,\dots,M^m}(s_j^m
- s_{j-1}^m) \to 0$ as $m\to \infty$  such that has the convergence
\eqref{eq:per_en4} and the stability
\eqref{eq:stab1}  holds at each time  $(s^m_{j-1})$
for $j=1,\dots,M^m $ for $m\in \Nz$. By
testing \eqref{eq:stab1} at time $(s^m_{j-1})$ on $(\hat y ,\hat
P)=(\UUU \tilde y \EEE (s^m_{j}),P(s^m_{j}))$
we obtain
\begin{align}
  &\int_\Omega W(\nabla  y(s^m_{j-1}) P^{-1}(s^m_{j-1}), P(s^m_{j-1}),\nabla
                                  P(s^m_{j-1}))\,\d x-\langle l(s^m_{j-1}),
  \UUU \tilde y\EEE(s^m_{j-1})\rangle \nonumber\\
  &\quad \leq \int_\Omega W(\nabla  y(s^m_{j}) P^{-1}(s^m_{j}), P(s^m_{j}),\nabla
                                  P(s^m_{j}))\,\d x-\langle l(s^m_{j}),
   \UUU \tilde y \EEE(s^m_{j})\rangle  \nonumber\\[2mm]
  &\qquad +\la  l(s^m_{j})-l(s^m_{j-1}), \UUU \tilde y\EEE(s^m_{j}) \ra +D(K\nabla y(s^m_{j-1}), P(s^m_{j-1}),P(s^m_{j}))
    .\nonumber
\end{align}
Taking the sum for $j=1,\dots,M^m$ we have
\begin{align}
 &\int_\Omega W(\nabla  y_0 P^{-1}_0, P_0,\nabla
                                  P_0)\,\d x-\langle l(0),
    y_0\rangle  \nonumber\\
  &\quad \leq \int_\Omega W(\nabla  y(t) P^{-1}(t), P(t),\nabla
                                  P(t))\,\d x-\langle l(t),
    \UUU \tilde y\EEE(t)\rangle \nonumber\\
  &\qquad + \sum_{j=1}^{M^m}D(K\nabla y(s^m_{j-1}),
    P(s^m_{j-1}),P(s^m_{j}))+\sum_{j=1}^{M^m}\la
    l(s^m_{j})-l(s^m_{j-1}), \UUU \tilde y\EEE(s^m_{j})
\ra. \label{eq:per_en5}
\end{align}
It suffices now to follow the argument in Subsection
\ref{sec:lower} leading to relations
\eqref{eq:per_en60}--\eqref{eq:per_en61} in order to pass to the
$\limsup$ as $m \to \infty$ in estimate \eqref{eq:per_en5}
and obtain 
\begin{align}
&\int_\Omega W(\nabla y_0P^{-1}_0, P_0,\nabla
                P_0)\,\d x -\langle l(0), y_0\rangle \nonumber\\
  &\quad \leq\int_\Omega W(\nabla \tilde y(t) P^{-1}(t), P(t),\nabla
                                  P(t))\,\d x-\langle l(t),\tilde
    y(t)\rangle   \nonumber\\
  &\qquad +{\rm
   Diss}_{[0,t]}(K\nabla y, P ) 
    -\int_0^t \la \Dot{l}(r),\UUU \tilde y\EEE(r)\ra \,\d r\nonumber
\end{align}
for a.e. $t\in (0,T)$. By combining this with \eqref{eq:upper_en} \UUU
for $y = \tilde y$ a.e., \EEE the
energy balance \UUU \eqref{eq:en100} \EEE follows.

\section{ Proof of Proposition \ref{prop:correspondence}:
  correspondence of solutions}\label{sec:correspondence}

 The first part of the statement follows by concatenating the arguments
of the proofs of Theorem \ref{thm:existencemv} and Proposition
\ref{lemma:functions}. Indeed, let $(y,P,\nu)$ be an energetic
measure-valued solution and $(\ove y_{n},\ove P_{n})$ the underlying
sequence of discrete solutions such that convergences
\eqref{6.14}--\eqref{eq:HellyA} and \eqref{eq:01}--\eqref{eq:03}, see
Section \ref{sec:convergence}. By further extracting from the same
sequence and following the
argument of Section \ref{proof:lemma:functions} we find $\tilde y :
[0,T]\to {\mathcal Y}$ such that $(y,P,\tilde y) $ is an energetic
function-valued solution.

Likewise, let $(y,P,\tilde y) $ be an energetic
function-valued solution   and indicate by  $(\ove
y_{n},\ove P_{n})$ the underlying sequence of  discrete solutions, see Section
\ref{proof:lemma:functions}. Starting from such sequence, one can
further extract and 
follow the argument of Section \ref{sec:convergence}--\ref{sec:lower}
and find an energetic measure-valued solution of the form $(y,P,\nu)$.

Let now $(y,P,\nu)$ and $(y,P,\tilde y)$ be an energetic
measure-valued and an energetic function-valued solution as above. In
particular, assume that these two are generated by the same sequence of
discrete solutions $(\ove y_{n},\ove P_{n})$. 
We are left with proving equality \eqref{eq:rela}.   Using
stability \eqref{def:stability} of the energetic measure-valued solution
 with  $(\hat y,\hat P)=(\tilde y(t),
P(t))$ we bound the measure-valued energy in terms of the energy in
$(\tilde y(t), P(t))$, namely,
\begin{equation}
    \la\!\la\nu_{t},W\ra\!\ra-\langle l(t), y(t)  \rangle \leq \int_\Omega W(\nabla \tilde y(t){P}^{-1}(t), P(t),\nabla P(t))\,\d x-\langle l(t), \tilde y(t)\rangle, 
  \label{eq:PWvsMV}
\end{equation}
for a.e. $ t\in (0,T)$.

The reverse inequality holds, as well. To see
this, we  fix $t\in (0,T)$ such that $\delta$ from
\eqref{eq:delta} is continuous at $t$ (note that this holds a.e.). We  choose $(\hat y,\hat P)=(\ove y_{ n}(\hat t),\ove
P_{ n}(\hat t))$ in the stability
\eqref{eq:stab1}, integrate for $\hat t$ in $(t,t+\epsi)$, and divide by $\epsi$ getting
\begin{align*}
 \quad&\int_\Omega W(\nabla \tilde y(t) P^{-1}(t), P(t),\nabla
  P(t))\,\d x-\langle l(t),\tilde y(t) \rangle \nonumber\\
  &\quad \leq \frac{1}{\epsi}\int_t^{t+\epsi}\!\!\int_\Omega W(\nabla
    \ove y_{ n}(\hat t) {\ove P}^{-1}_{ n}(\hat t),\ove P _{ n}(\hat t),\nabla
    \ove P _{ n}(\hat t))\,\d x\, \d\hat
    t\nonumber\\
  &\qquad -\frac{1}{\epsi}\int_t^{t+\epsi}\langle l(\hat t),\ove
    y_{ n}(\hat t)\rangle \,\d\hat t  +\frac{1}{\epsi}\int_t^{t+\epsi} {\mathcal D}
    (K\nabla y(t),  P(t),\ove P_{ n}(\hat t))\,\d\hat t. 
\end{align*}
Arguing as above, by passing to the limit first as $ n \to \infty$ and
than as $\epsi \to 0$,  and using that
\begin{align*}
  &\lim_{\epsi \to 0}\lim_{n\to \infty}\frac{1}{\epsi}\int_t^{t+\epsi} {\mathcal D}
    (K\nabla y(t),  P(t),\ove P_{ n}(\hat t))\,\d\hat t = \lim_{\epsi \to 0}\frac{1}{\epsi}\int_t^{t+\epsi} {\mathcal D}
  (K\nabla y(t),  P(t), P(\hat t))\,\d\hat t\\
  &\quad \leq \lim_{\epsi \to 0} \lim_{\epsi \to
    0}\frac{1}{\epsi}\int_t^{t+\epsi} {\rm Diss}_{[t,\hat t]}(K\nabla
    y,P) \, \d \hat t \\
  &\quad = \frac{1}{\epsi}\int_t^{t+\epsi} \big(
    \delta(\hat t)- \delta(t) \big) \, \d \hat t \leq \lim_{\epsi \to
    0} \delta(t+\epsi) - \delta (t)=0,
  \end{align*}
 we obtain that 
\begin{align*}
    \int_\Omega W(\nabla \tilde y(t){P}^{-1}(t), P(t),\nabla P(t))\,\d x-\langle l(t), \tilde y(t)\rangle\leq \la\!\la\nu_{t},W\ra\!\ra-\langle l(t),y(t) \rangle,
\end{align*}
for a.e. $t\in (0,T)$. Combined with \eqref{eq:PWvsMV}, this gives
equality \eqref{eq:rela} and concludes the proof.

 \section*{Acknowledgements} This research was funded in
whole or in part by the Austrian Science Fund (FWF) projects 10.55776/F65,  10.55776/I5149,
10.55776/P32788, by the OeAD-WTZ project CZ 09/2023, \UUU and by the
M\v SMT project ERC CZ No. LL2310. \EEE For
open access purposes, the authors have applied a CC BY public copyright
license to any author-accepted manuscript version arising from this
submission.   We thank Martin
Kru\v z\'\i k for pointing \UUU out some relevant references \EEE and for
proposing to discuss linearization, Section
\ref{sec:linearization}. 
 

\appendix

\section{Helly selection principle}

We present a variant of the classical Helly selection principle, which
is used in the compactness proof in Section \ref{sec:convergence}
above. With respect to \cite[Theorem B.5.13]{Mielke2015}, the
dissipation is here time-dependent. We present a lemma in some
generality, by suitably extending the setting of Section
\ref{sec:assumptions} but without introducing new notation. In
particular, for the purposes of this appendix we assume that
\begin{align}
  &\mathcal{F} \ \text{and} \  \mathcal{P} \ \text{are Hausdorff
  topological spaces}, \label{eq:h1}\\
  &\mathcal{D}: \mathcal{F} \times \mathcal{P} \times \mathcal{P} \to
  (0,\infty],\label{eq:h2}
\end{align}

For given trajectories $(\haz F,P):[0,T]\to \mathcal{F} \times
\mathcal{P} $, for all $s,\, t \in [0,T]$
we let
\begin{align*}
  {\rm Diss}_{[s,t]}(\haz F,P) = \sup\left\{\sum_{i=1}^N \mathcal{D}(\haz
  F(t_i), P (t_i), P(t_{i-1})) \ : \ 0=t_0<t_1<\dots<t_N=T \right\}
\end{align*}
where the supremum is taken on all partitions of $[0,T]$.
We assume the following
\begin{align}
  & \text{For all sequentially compact $K\subset \mathcal{P}$}, \
    \forall \haz F \in \mathcal{F} \quad \text{we have:}\nonumber\\
  &  (\haz F_k \stackrel{\mathcal{F}}{\to} \haz F,  \ \ P_k \in K, \ \
    \text{and}\ \ \max\{\mathcal{D}(\haz
    F_k,P_k,P), \mathcal{D}(\haz F_k,P,P_k)\} \to 0 )\nonumber \\
  &\qquad \Rightarrow
    \ \  P_k \stackrel{\mathcal{P}}{\to} P, \label{eq:h4}\\[2mm]
  &  (\haz F_k \stackrel{\mathcal{F}}{\to} \haz F, \ \ P_k
    \stackrel{\mathcal{P}}{\to} P, \ \ \haz P_k
    \stackrel{\mathcal{P}}{\to} \haz P, \ \ P,\,P_k \in K)\ \ 
   \nonumber  \\[1mm]
  &\qquad \Rightarrow \ \ 
      \mathcal{D}(\haz F,P,\haz P) \leq \liminf_{k\to \infty} \mathcal{D}(\haz F_k,P_k,\haz P_k).\label{eq:h5} 
\end{align}

\begin{theorem}[Extended Helly selection principle]\label{thm:Helly}
Assume \eqref{eq:h1}--\eqref{eq:h5} and let $K $ be sequentially
compact in $\mathcal{P}$. Let $\haz F_k:[0,T] \to \mathcal{F}$ and
$P_k:[0,T] \to \mathcal{P}$ be such that
\begin{align}
  & \haz F_k(t) \stackrel{\mathcal{F}}{\to}  \haz F(t) \quad \forall t\in
    [0,T],\label{eq:hh1}\\[1mm]
& P_k(t)\in K \quad \forall\, k\in\Nz, \ \forall t\in[0,T],\label{eq:hh3}\\
  & \sup_{k\in\Nz} 
\mathrm{Diss}_{[0,T]}(\haz F_k, P_k)<\infty. \label{eq:hh4}
\end{align}
Then, there exist a subsequence $(P_{k_\ell})_{\ell\in\Nz}$, a 
function $P:[0,T] \to \mathcal{P}$, and a nondecreasing function $\delta:[0,T] \to [0,\infty)$ with the following properties:
\begin{align}
& 
\mathrm{Diss}_{[0,t]}(\haz
                F_{k_\ell},P_{k_\ell}) \to \delta(t)\quad \forall t\in[0,T], \label{eq:hh3}\\
& {P}_{k_\ell}(t) \stackrel{\mathcal{P}}{\to} P(t) \quad \forall t\in[0,T],\label{eq:hh4}\\[1mm]
&
\mathrm{Diss}_{[s,t]}(\haz F,P)\leq\delta(t)-\delta(s) \quad \forall 0\leq
                                                                                             s\leq
                                                                                             t
                                                                                             \leq
                                                                                             T. \label{eq:hh5}
\end{align}
\end{theorem}

\begin{proof}
  The functions $\delta_k:[0,T]\to[0,\infty)$ given by $\delta_k(t) =
  \mathrm{Diss}_{[0,t]}(\haz F_{k},P_{k})$ are
  nondecreasing. By the classical Helly principle one can extract a
  not relabeled subsequence such that $\delta_k(t) \to \delta(t)$ for
  all $t\in[0,T]$, where $\delta$ is nondecreasing, so that
  \eqref{eq:hh3} holds.

 The function $\delta$ has at most countably many discontinuity points
 $J\subset [0,T]$.  Let $A\subset [0,T]$ be countable
 and dense in $[0,T]$ with $J  \subset A$. For each $t\in
 A$ one has $P_k(t) \in K$ from \eqref{eq:hh3} and one can extract a subsequence, possibly
 depending on $t$, such that $P_{k_j^t}(t)$ converges. We define
 $P(t):=\lim_{j\to \infty} P_{k_j^t}(t)$. As $A$ is countable, by a
 standard diagonal extraction argument we find a subsequence
 $P_{k_\ell}$ such that $P_{k_\ell}(t) \to P(t)$ for all $t\in A$.

 Fix now $t \in [0,T]\setminus A$. As $P_{k_\ell}(t)\in K$ one can
 extract without relabeling in such a way that $P_{k_\ell}(t)\to
 P_*$. We now prove that, for all $(t_n)\in A$ with $t_n \to t$ one
 has $P(t_n)\to P_*$. Assume first that $t_n \leq t$. As
 $\haz F_k(t_n)\to \haz F(t_n)$ from \eqref{eq:hh1} we can use the
 lower semicontinuity
 \eqref{eq:h5} and get
 \begin{align*}
   &\mathcal{D}(\haz F(t_n),P(t_n),P_*)\leq \liminf_{\ell\to
   \infty} \mathcal{D}(\haz
   F_{k_\ell}(t_n),P_{k_\ell}(t_{n}),P_{k_\ell}(t))\\
   &\quad \leq \liminf_{\ell\to
   \infty} \mathrm{Diss}_{[t_{n},t]}(\haz
     F_{k_\ell},P_{k_\ell})\leq \liminf_{\ell\to
   \infty} (\delta_{k_\ell}(t)-\delta_{k_\ell}(t_n)) = \delta(t)-\delta(t_n).
 \end{align*} If
 $t_n \geq t$, by
 exchanging the roles of $P(t_n)$ and $P_*$ we deduce that
 $$\mathcal{D}(\haz F(t_n),P_*,P(t_n))\leq \delta(t_n) - \delta(t).$$
 This entails that 
$$\max \{\mathcal{D}(\haz F(t_n),P(t_n),P_*), \mathcal{D}(\haz
F(t_n),P_*,P(t_n))\} \leq |\delta(t)-\delta(t_n)|.$$
 As $\delta$ is continuous in $t$, by passing to the limit as $n\to
 \infty$ we have that $ |\delta(t)-\delta(t_n)|\to 0$ and assumption \eqref{eq:h4} implies that $P(t_n)\stackrel{\mathcal{P}}{\to} P_*$.
 This in particular ensures that $P_*$ is uniquely
 determined and can hence be used to define $P(t):=P_*$, so that
 $P_{k_\ell}(t)\to P(t)$ for all $t\in [0,T]$, which is \eqref{eq:hh4}.

 Let $\{s=t=_0<t_1<t_N=t\}$ be any partition of the interval
 $[s,t]\subset [0,T]$. We compute
 \begin{align*}
   &\sum_{i=1}^N \mathcal{D}(\haz F(t_i),P(t_i),P(t_{i-1})) \leq
   \liminf_{\ell\to \infty}  \sum_{i=1}^N \mathcal{D}(\haz
     F_{k_\ell}(t_i),P_{k_\ell}(t_i),P_{k_\ell}(t_{i-1})) \\
   &\quad \leq
   \liminf_{\ell\to \infty}  \mathrm{Diss}_{[s,t]}(\haz F_{k_\ell},P_{k_\ell}) =
   \delta(t)- \delta(s).
 \end{align*}
 By passing to the supremum with respect to partitions of $[s,t]$ one
 obtains \eqref{eq:hh5}.
\end{proof}


\section{Riemann sums}

We present a straightforward extension of \cite[Lemma
4.12]{DalMaso2004} on Riemann sums, used in Subsection
\ref{sec:lower} and in Section \ref{proof:lemma:functions}.

\begin{lemma}[Riemann sums]\label{lemma:DalMaso} Let $f\in L^1(0,T;E)$ where $(E, \| \cdot \|)$ is a
  Banach space and let $N \subset (0,T)$ be a null set. Then, there exists
  a sequence of partitions $\{0=s_0^m<\dots<s^m_{M^m}=T\}$ with
  $\max_{j=1,\dots,M^m}(s_j^m-s_{j-1}^m) \to 0$ as $m\to \infty$ such
  that $s_j^m \not \in N$ for all $j=1,\dots,M^m-1$ and
  \begin{equation} \lim_{m\to \infty}\sum_{j=1}^{M^m} \int_{s_{j-1}^m}^{s_j^m} \|
  f(s_j^m) - f(t)\|\, \d t =0.\label{eq:dal0}
  \end{equation}
  In particular, we have that
  $$  \sum_{j=1}^{M^m} (s_{j}^m - s_{j-1}^m)
  f(s_j^m) \to \int_0^T f(t)\, \d t$$
  strongly in $E$ as $m\to \infty$.
\end{lemma}

\begin{proof} By taking $N = \emptyset$ the statement corresponds to \cite[Lemma
4.12]{DalMaso2004}, which is proved in \cite{DalMaso2004bis} by arguing
on any quasiuniform partition $\{0=s_0^m< s_1^m<\dots,s_{M^m}^m=T\}$
of the following form 
\begin{equation}s_j^m = s^*+ \frac{j}{M^m} \quad \text{for} \ \  \rho_m \leq j \leq
\sigma_m.\label{eq:dal}
\end{equation}
Here, $\rho_m = \min\{z \in \Zz \ : \ 0<s+z/M^m\}$, $\sigma_m= \max\{z
\in \Zz \ : \  s+z/M^m<T\}$, and $s^*$ is any point in $[0,1]\setminus G$, where
$G$ is a null set depending on $f$.

We shall now check that the points of the partition can be chosen not to
belong the null set $N$. To this aim, we define the null sets
$$N_{m,j} = \left(N - \frac{j}{M^m}\right) \cap [0,1]\quad \text{for} \ \  m\in
\mathbb{N}, \ j \in \Zz$$
which are obtained by simply translating $N$ and intersecting it with $[0,1]$. Next we define 
$$\ove G =  \left(\cup_{(m,j)\in \mathbb{N}\times \Zz} N_{m,j}\right)
\cup G$$ 
which is a countable union of null sets, hence again a null set. In particular, one can choose $s^*\in
[0,1]\setminus \ove G \subset [0,1]\setminus   G$, define
$(s_j^m)_{j=1}^{M^m-1} $ as  in \eqref{eq:dal}, and follow the proof
of \cite[Lemma
4.12]{DalMaso2004} from  \cite{DalMaso2004bis} in order to get
\eqref{eq:dal0}. Moreover, starting from such a $s^*$ we have that the
points $s_j^m=s^*+j/{M^m}$ for
all $m\in \mathbb{N}$ and $\rho_m \leq j \leq \sigma_m$ do not belong
to $N$, which proves the assertion.
\end{proof}

\section*{Data availability statement}

No original data has been generated by this research.

\bibliographystyle{acm}

\begin{thebibliography}{10}
  

\bibitem{Alibert}
  J.~J.~ Alibert, G.~Bouchitt\'e. Non-uniform integrability and
  generalized Young measures. {\it J. Convex Anal.} 4 (1997), no. 1,
  129--147.
  

\bibitem{ams}
  F. Auricchio, A. Mielke,  U. Stefanelli.
 A rate-independent model for the isothermal quasi-static
 evolution of shape-memory materials.
 {\em Math. Models Meth. Appl. Sci.} 18 (2008), 125--164.

\bibitem{Mora}
 J.-F.  Babadjian, G. A. Francfort, M. G. Mora. Quasi-static evolution
 in nonassociative plasticity: the cap model. {\it SIAM J. Math. Anal.} 44 (2012), no. 1, 245-–292. 

\bibitem{Brenier2011}
 Y. Brenier,  C. De~Lellis, L.  Sz{\'e}kelyhidi.
\newblock Weak-strong uniqueness for measure-valued solutions.
\newblock {\em Comm. Math. Phys.} 305 (2011), 351--361.

\bibitem{Brezis73}
H. Brezis.
\newblock {\em Operateurs maximaux monotones et semi-groupes de contractions
  dans les espaces de Hilbert}.
\newblock Elsevier, 1973.

\bibitem{Conti1}
  S.~Conti, A.~Garroni, A.~Massaccesi. Modeling of dislocations and relaxation of functionals on 1-currents
with discrete multiplicity. {\it  Calc. Var. Partial Differential
  Equations}, 54 (2015), 1847--1874.

\bibitem{Conti2}
 S.~Conti, A.~Garroni, M.~Ortiz. The line-tension approximation as the
 dilute limit of linear-elastic dislocations. {\it
Arch. Ration. Mech. Anal.} 218 (2015), 699--755.

\bibitem{Conti3}
S.~Conti, F.~Theil. Single-slip elastoplastic microstructures. {\it Arch. Ration. Mech. Anal.} 178 (2005), 125--148.


\bibitem{DalMaso1}
G. Dal Maso, A. DeSimone, F. Solombrino. Quasistatic evolution for
Cam-Clay plasticity: a weak formulation via viscoplastic
regularization and time rescaling. {\it Calc. Var. Partial
  Differential Equations}, 40 (2011),
125--181.

\bibitem{DalMaso2}
G. Dal Maso, A. DeSimone,  F. Solombrino. Quasistatic evolution
for Cam-Clay plasticity: properties of the viscosity
solutions. {\it Calc. Var. Partial
  Differential Equations},  44 (2012), 495--541.

\bibitem{DalMaso2004}
G.  Dal Maso, G. A. Francfort, R.  Toader. Quasistatic crack growth in
nonlinear elasticity. {\it Arch. Ration. Mech. Anal.} 176 (2005),
no. 2, 165--225.


\bibitem{DalMaso2004bis}
G.  Dal Maso, G. A. Francfort, R.  Toader. Quasistatic crack growth in
nonlinear elasticity. arXiv:0401196, 2004.

\bibitem{DalMaso3}
G.  Dal Maso, F. Solombrino. Quasistatic evolution for Cam-Clay
plasticity: the spatially homogeneous case. {\it
  Netw. Heterog. Media}, 5 (2010), no. 1, 97--132.




\bibitem{Davoli}
E. Davoli, G. A. Francfort. A critical revisiting of finite
elasto-plasticity. {\it SIAM J. Math. Anal.} 47 (2015), 526--565.


\bibitem{davoli.mora}
E. Davoli, M.~G. Mora. A quasistatic evolution model for perfectly plastic plates derived by
  {$\Gamma$}-convergence. {\it Ann. Inst. H. Poincar\'e Anal. Non
Lin\'eaire}, 30 (2013), 615--660.

\bibitem{Dillon}
  O. W. Dillon, J. Kratochv\'\i l. A strain gradient theory of
  plasticity. {\it Int. J. Solid Struct.} 6 (1970), 1513--1533.
  
  
  \bibitem{DiPerna}
 R.~J.~DiPerna, A.~J.~Majda. Oscillations and concentrations
  in weak solutions of the incompressible fluid
  equations. {\it Comm. Math. Phys.} 108 (1987), no. 4, 667--689.
  

\bibitem{Fleck}
N. A. Fleck, J. W. Hutchinson. Strain gradient plasticity.
{\it Adv. Appl. Mech.} 33 (1997), 295--361.

\bibitem{Fleck2}
  N. A. Fleck, J. W. Hutchinson. A reformulation of strain gradient
  plasticity. {\it J. Mech. Phys. Solids},
  49 (2001), 2245--2271.

\bibitem{Fonseca}
  I.   Fonseca, S. M\"uller, P. Pedregal. Analysis of concentration
  and oscillation effects generated by gradients. {\it SIAM J. Math. Anal.} 29 (1998), no. 3, 736--756. 
 

\bibitem{Mora2}
 G. A.  Francfort, M. G. Mora. Quasistatic evolution in
  non-associative plasticity revisited. {\it Calc. Var. Partial
  Differential Equations}, 57 (2018), no. 1, Paper No. 11, 22 pp.

\bibitem{Francfort}
 G.~A.~Francfort, U. Stefanelli. Quasi-static evolution for the
 Armstrong-Frederick hardening-plasticity model. {\it
   Appl. Math. Res. Express AMRX}, 2 (2013), 297--344.

\bibitem{Garroni}
  A.~Garroni, G.~Leoni, M.~Ponsiglione. Gradient theory for plasticity
  via homogenization of discrete dislocations. {\it
    J. Eur. Math. Soc. (JEMS)}, 12 (2010), 1231--1266.

\bibitem{Ginster}
  J. Ginster. Plasticity as the $\Gamma$-limit of a two-dimensional dislocation energy: the critical regime without the
assumption of well-separateness. {\it Arch. Ration. Mech. Anal.} 233
(2019), 1253--1288.

\UUU
\bibitem{Ginster1}
J.~Ginster.
\newblock Strain-gradient plasticity as the $\Gamma$-limit of a nonlinear dislocation energy with mixed growth.
\newblock {\em SIAM J. Math. Anal.} 51 (2019), no.~4, 3424--3464.
\EEE

\bibitem{cplas1}
  D. Grandi, U. Stefanelli. Finite plasticity in $P^\top\!P$. Part I: constitutive model. {\it Contin. Mech. Thermodyn.}  29 (2017), 97--116.

\bibitem{cplas2}
  D. Grandi, U. Stefanelli. Finite plasticity in $P^\top P$. Part II:
  quasistatic evolution and linearization. {\it SIAM
  J. Math. Anal.} 49 (2017), 1356--1384.

  \bibitem{Grandi2}
D. Grandi, U. Stefanelli. Existence and linearization for the
Souza-Auricchio model at finite strains. {\it
Discrete Contin. Dyn. Syst. Ser. S}, 10 (2017), 1257--1280.

  \bibitem{Gurtin}
  M. E. Gurtin. On the plasticity of single crystals: free energy, microforces, plastic-strain gradients.
  {\it J. Mech. Phys. Solids}, 48 (2000), 989--1036.

  \bibitem{Gurtin2}
 M. E. Gurtin, L. Anand. A theory of strain-gradient plasticity for
 isotropic, plastically irrotational materials. I. Small deformations.
 {\it J. Mech. Phys. Solids}, 53 (2005),
1624--1649.

\bibitem{Gurtin2010}
M. E.  Gurtin, E.  Fried, L.  Anand.
\newblock {\em The mechanics and thermodynamics of continua}.
\newblock Cambridge University Press, 2010.


\bibitem{Rindler1}
T.~Hudson, F.~Rindler. Elasto-plastic evolution of crystal materials
driven by dislocation flow.  {\it Math. Models Methods Appl. Sci.} 32 (2022), 851--910.


\bibitem{Bogdan2020}
J. Kristensen, B. Rai{\c{t}}{\u{a}}.
\newblock An introduction to generalized Young measures.
\newblock Preprint Max-Planck-Institut f{\"u}r Mathematik in den
Naturwissenschaften Leipzig, 45, (2020).


\bibitem{Kristensen}
J.~Kristensen, F.~Rindler. Characterization of generalized gradient
Young measures generated by sequences in $W^{1,1}$ and $BV$. {\it
  Arch. Ration. Mech. Anal.} 197 (2010), no. 2, 539--598.


\bibitem{Kroener}
  E. Kr\"oner. Allgemeine Kontinuumstheorie der Versetzungen und
  Eigenspannungen. {\it Arch. Rational Mech. Anal.} 4 (1959), 273--334.

\bibitem{compos2}
   M. Kru\v z\'\i k, D. Melching, U. Stefanelli.
   Quasistatic evolution for dislocation-free finite plasticity. {\it ESAIM
     Control Optim. Calc. Var.} 26 (2020), Paper No. 123, 23 pp.

 \bibitem{Kupferman}
   R.~Kupferman. C.~Maor. The emergence of torsion in the continuum
   limit of distributed edge-dislocations. {\it J. Geom. Mech.} 7 (2015), 361--387.

 \bibitem{Laborde1}
   P. Laborde. A nonmonotone differential inclusion. {\it Nonlin. Anal.} 11
   (1986),   757--767.

 \bibitem{Laborde2}
   P. Laborde. Analysis of the strain-stress relation in plasticity
   with non-associative laws. {\it Internat. J. Engrg. Sci.} 25
   (1987),   655--666.


\bibitem{Lee}
E. Lee.
\newblock Elastic-plastic deformation at finite strains.
\newblock {\em J. Appl. Mech.} 36 (1969), 1--6.

\bibitem{Mainik}
  A. Mainik, A. Mielke. Global existence for rate-independent gradient
  plasticity at finite strain. {\it
  J. Nonlinear Sci.} 19 (2009), 221--248.
  
\bibitem{Mandel1972}
J. Mandel.
\newblock {\it Plasticit{\'e} classique et viscoplasticit{\'e}}. CISM courses and
  lectures, vol. 97, Springer-Verlag, Berlin, 1972.

\bibitem{Damage2019}
D. Melching, R.  Scala, J.  Zeman.
\newblock Damage model for plastic materials at finite strains.
\newblock {\em ZAMM Z. Angew. Math. Mech.} 9 (2019),
e201800032.



\bibitem{Mielke2002}
 A. Mielke. Finite elastoplasticity Lie groups and geodesics on
 $SL(d)$.  {\it Geometry, mechanics, and dynamics}, 61--90, Springer, New York, 2002. 

\bibitem{Mielke2003}
 A. Mielke.
\newblock Energetic formulation of multiplicative elasto-plasticity using
  dissipation distances.
  \newblock {\em Contin. Mech. Thermodyn.} 15 (2003), 351--382.

  
\bibitem{Mielke04b}
A.~Mielke. Existence of minimizers in incremental
elasto–plasticity with finite strains. {\it SIAM J. Math. Anal.}
36 (2004), 384--404.

\bibitem{Mielke2005a}
  A. Mielke. Evolution in rate-independent systems. {\it Evolutionary
    equations. Vol. II}, 461--559, Handb. Differ. Equ.,
  Elsevier/North-Holland, Amsterdam, 2005. 

\bibitem{Mielke2005}
A Mielke,  S.  M{\"u}ller.
\newblock Lower semi-continuity and existence of minimizers in incremental
  finite-strain elastoplasticity.
  \newblock {\em  ZAMM Z. Angew. Math. Mech.} 86 (2006), 233--250.

 
\bibitem{MRS_vis}
 A.~Mielke, R.~Rossi, G.~Savar\'e. Global existence results for
 viscoplasticity at finite strain. {\it Arch. Ration. Mech. Anal.} 227 (2018), no. 1, 423--475.
 



\bibitem{Mielke2015}
A. Mielke,  T.  Roub{\'\i}{\v{c}}ek.
\newblock {\it Rate-independent systems. Theory and application}. Applied Mathematical Sciences, 193. Springer, New York, 2015.

\bibitem{Mielke16}
A. Mielke,  T.  Roub{\'\i}{\v{c}}ek.
  Rate-independent elastoplasticity at finite strains and its
  numerical approximation. {\it Math. Models Methods Appl. Sci.} 26
  (2016), 2203--2236.

  

\bibitem{ms5}
 A.~Mielke, U.~Stefanelli. Linearized plasticity is the
  evolutionary $\Gamma$-limit of finite plasticity. {\it J. Eur. Math. Soc. (JEMS)},
  15 (2013), no. 3, 923--948.


\bibitem{Aifantis}
  H.-B. M\"uhlhaus, E. Aifantis. A variational principle for gradient
  plasticity. {\it  Int. J. Solids Struct.} 28
  (1991), 845--857.

  \UUU
  
\bibitem{Scardia2}
S.~M{\"u}ller, L.~Scardia, C.~I. Zeppieri.
\newblock Geometric rigidity for incompatible fields, and an application to strain-gradient plasticity.
\newblock {\em Indiana Univ. Math. J.}
63 (2014), no.~5, 136--1396. 
\EEE

\bibitem{Coleman1974}
W. Noll, B. D. Coleman.
\newblock The thermodynamics of elastic materials with heat conduction and
viscosity. {\it Arch. Rational Mech. Anal.} 13 (1963), 167--178.

\bibitem{Rindler2}
F.~Rindler. Energetic solutions to rate-independent large-strain
elasto-plastic evolutions driven by discrete dislocation flow. {\it
  J. Eur. Math. Soc. (JEMS)}, to appear (2024). arXiv:2109.14416.

\bibitem{Roeger}
  M. R\"oger, B. Schweizer. Strain gradient visco-plasticity with dislocation densities contributing to the
  energy. {\it Math. Models Methods Appl. Sci.} 27 (2017), 2595--2629.

  \UUU
  
\bibitem{Scardia1}
L.~Scardia, C.~I. Zeppieri.
\newblock Line-tension model for plasticity as the $\Gamma$-limit of a nonlinear dislocation energy.
\newblock {\em SIAM J. Math. Anal.} 44 (2012), no.~4, 2372--2400.

\EEE

\bibitem{Stefanelli2001}
U. Stefanelli.
\newblock Well-posedness and time discretization of a nonlinear Volterra
  integrodifferential equation.
\newblock {\em  J. Integral Equations Appl.} 13 (2001), no. 3, 273--304. 

\bibitem{compos}
  U. Stefanelli. Existence for dislocation-free finite plasticity. {\it ESIAM Control Optim. Calc. Var.} 25 (2019),
  Art. 21, 20 pp.
  
\bibitem{Ulloa2021}
J. Ulloa, R.  Alessi, J. Wambacq, G.  Degrande, S. Francois.
\newblock On the variational modeling of non-associative plasticity.
\newblock {\em Int. J. Solids Struct.} 217 (2021),
  272--296.

\bibitem{WiedT2012}
E. Wiedemann.
\newblock {\em Weak and measure-valued solutions of the incompressible Euler
  equations}.
\newblock PhD thesis, Universit{\"a}ts-und Landesbibliothek Bonn, 2012.


\end{thebibliography}

\end{document}